\documentclass[final,1p,11pt,times]{elsarticle}
\usepackage[utf8]{inputenc}
\usepackage[dvipsnames]{xcolor}
\usepackage[shortlabels]{enumitem}
\usepackage{graphicx}
\usepackage{amsmath,amssymb,amsthm,mathtools}
\usepackage{latexsym, verbatim}
\usepackage{appendix}
\usepackage{units}
\usepackage{amsmath}
\usepackage[normalem]{ulem}
\usepackage[hidelinks]{hyperref}
\usepackage{algorithm}
\usepackage{algpseudocode}
\usepackage{hyperref} 
\usepackage{soul}
\usepackage{nicematrix}
\usepackage{pgfplots}
\pgfplotsset{compat=1.18}
\usetikzlibrary{pgfplots.groupplots}

\mathtoolsset{showonlyrefs}
\newtheorem{theorem}{Theorem}[section]
\newtheorem{corollary}[theorem]{Corollary} 
\newtheorem{lemma}[theorem]{Lemma}
\newtheorem{proposition}[theorem]{Proposition}
\newtheorem{question}[theorem]{Question}

\newtheorem{remark}[theorem]{Remark}

\theoremstyle{definition}
\newtheorem{example}[theorem]{Example}
\newtheorem{definition}[theorem]{Definition}

\newcommand{\fa}{\mathbf{a}}

\newcommand{\fx}{\mathbf{x}}

\newcommand{\Ucal}{\mathcal{U}}

\newcommand{\Vcal}{\mathcal{V}}

\newcommand{\CC}{\mathbb{C}}
\newcommand{\NN}{\mathbb{N}}

\newcommand{\PP}{\mathbb{P}}
\newcommand{\RR}{\mathbb{R}}

\newcommand{\Ncal}{\mathcal{N}}

\DeclareMathOperator{\SV}{SV}
\DeclareMathOperator{\codim}{codim}
\DeclareMathOperator{\Gr}{Gr}
\DeclareMathOperator{\Sym}{Sym}
\DeclareMathOperator{\Span}{Span}
\DeclareMathOperator{\rank}{rank}

\DeclareMathOperator{\Sing}{Sing}

\DeclareMathOperator{\Vect}{Vect}

\DeclareMathOperator{\Vol}{Vol}
\DeclareMathOperator{\GL}{GL}
\newcommand{\T}{^\mathsf{T}}

\newcommand{\Zcal}{\mathcal{Z}}

\makeatletter
\def\ps@pprintTitle{%
 \let\@oddhead\@empty
 \let\@evenhead\@empty
 \let\@oddfoot\@empty
 \let\@evenfoot\@empty
}
\makeatother
\begin{document}


\begin{frontmatter}

\title{{\bf A Real Generalized Trisecant Trichotomy}}

\author[label1]{Kristian Ranestad}
\affiliation[label1]{organization={Department of Mathematics, University of Oslo}, country = {Norway}}

\author[label2]{Anna Seigal}

\author[label2]{Kexin Wang}

\affiliation[label2]{organization = {John A. Paulson School of Engineering and Applied Sciences, Harvard University}, country = {United States}}

\begin{abstract}
The classical trisecant lemma says that a general chord of a non-degenerate space curve is not a trisecant; that is, the chord only meets the curve in two points. The generalized trisecant lemma extends the result to higher-dimensional varieties. It states that the linear space spanned by general points on a projective variety intersects the variety in exactly these points, provided the dimension of the linear space is smaller than the codimension of the variety and that the variety is irreducible, reduced, and non-degenerate. We prove a real analog of the generalized trisecant lemma, which takes the form of a trichotomy. Along the way, we characterize the possible numbers of real intersection points between a real projective variety and a complementary dimension real linear space. We show that any integer of correct parity between a minimum and a maximum number can be achieved. We then specialize to Segre-Veronese varieties, where our results apply to the identifiability of independent component analysis,  tensor decomposition, and typical tensor ranks.
\end{abstract}


\end{frontmatter}

\section{Introduction}
A trisecant is a line that meets a variety in three points.
The classical trisecant lemma says that if $X$ is a non-degenerate irreducible curve in $\PP_\CC^3$, then the variety of trisecants has dimension one in the Grassmannian of lines in $\PP_\CC^3$, which we denote by $\Gr(1,3)$. 
Hence a general chord of $X$ is not a trisecant, 
since the variety of chords has dimension two in $\Gr(1,3)$.
The trisecant lemma has been generalized in various ways. 
We consider a generalization to higher-dimensional varieties. 
Recall that a variety is non-degenerate if not contained in a hyperplane.

\begin{theorem}[A Generalized Trisecant Lemma, see {\cite[Proposition 2.6]{chiantini2002weakly}}]\label{lemma:trisecant}
Let $X \subseteq \PP_\CC^{N-1}$ be an irreducible, reduced, non-degenerate projective variety of dimension $d$ and let $n$ be a positive integer with $n+d < N$.
Let $P_1,\ldots,P_n$ be general points on $X$. Then the intersection of $X$ with
the subspace spanned by $P_1,\ldots,P_n$ consists only of the points $P_1,\ldots,P_n$.
\end{theorem}

The generalized trisecant lemma can be restated as the following trichotomy.

\begin{theorem}[Reformulation of Theorem \ref{lemma:trisecant}]
\label{thm:complex_trichotomy}
    Let $X\subseteq \PP_\CC^{N-1}$ be an irreducible, reduced, non-degenerate projective variety of dimension $d$. Let $P_1,\ldots,P_n$ be general points on $X$ and let $W$ be the projective linear space they span.
Then 
\begin{enumerate}[(a)]
    \item If $n+d<N$, then $X\cap W =\{P_1,\ldots,P_n\}$.
    \item  If $n+d=N$, then $\deg X\geq n$. 
    When $\deg X>n$, $X\cap W \supsetneqq\{P_1,\ldots,P_n\}$. 
    When $\deg X = n$, $X\cap W =\{P_1,\ldots,P_n\}$ and 
    $X$ is called a variety with minimal degree; it is either a quadric hypersurface, a cone over the Veronese surface, or a rational normal scroll.
    \item If $n+d>N$, then $X\cap W \supsetneqq \{P_1,\ldots,P_n\}.$
\end{enumerate}
\end{theorem}
\begin{proof}
    The case $n+d<N$ is the generalized trisecant lemma.
        When $n+d=N$, the degree of $X$ is at least $N-d=n$, since
        $\deg X$ is the number of intersection points between $X$ and a generic linear space of dimension $n-1$ and $X$ is non-degenerate so 
        the intersection points span the linear space. 
    When $\deg X>n$, the intersection $X \cap W$ contains points other than $P_1,\ldots,P_n$. 
    When $\deg X = n$, the variety $X$ has minimal degree and the intersection of $X$ with $W$ is precisely $P_1,\ldots,P_n$. 
    For the classification of irreducible non-degenerate projective varieties with minimal degree, see e.g. \cite[Theorem 19.9]{harris2013algebraic}.   
    When $n+d>N$, the intersection between $X$ and $W$ has dimension at least $n-1+d-N>0$ so it contains infinitely many points. In particular, the intersection contains a point other than $P_1,\ldots,P_n$. 
\end{proof}

A tensor is a multidimensional array and a tensor decomposition writes a tensor as a sum of rank one tensors.
Suppose we have a tensor $T = \sum_{i=1}^r x_i$ where $x_1,\ldots,x_r$ are rank one tensors and that we can recover $V:=\Span\{x_1,\ldots,x_r\}$. 
The tensor decomposition is unique when the linear space $V$ intersects the variety of rank one tensors $X$ in precisely these $r$ points $x_1,\ldots,x_r$ and this is the content of the generalized trisecant trichotomy~\cite[Proposition 3.2]{kileel2025subspace}.
For many applications in statistics \cite{anandkumar2014tensor,mccullagh2018tensor,robeva2019duality,bi2021tensors} and data analysis \cite{gtex2015genotype,hore2016tensor,wang2025contrastive}, we are often only interested in real tensor decompositions.
So, a natural question is to find a real analog of the generalized trisecant trichotomy
by restricting the points we use to span the linear space to be real and asking if there is an extra real point in the intersection of the linear space and the variety.

We call a variety $X\subseteq \PP_\CC^{N-1}$ a \emph{real projective variety} if it is irreducible, reduced, non-degenerate, and can be written as the vanishing locus of some real homogeneous polynomials. We use \emph{$X_\RR$} to denote the collection of real points in $X$ with induced Euclidean topology. 
More precisely, we take the topology to be the quotient topology induced from the Euclidean topology on $\RR^N$ by making the quotient map $\RR^N \mapsto \PP_\RR^{N-1}$ continuous.
Similarly, we say a linear space is real if it is defined by real linear forms.
When we talk about the dimension of a linear space, we always mean its projective dimension.

It turns out that the real analog of Theorem \ref{thm:complex_trichotomy} depends on the set of possible numbers of real intersection points between $X$ and a complementary dimension real linear space.
Bounds on such numbers are studied in e.g. \cite{sturmfels1994number,okonek2014intrinsic,soprunova2006lower,hein2016congruence}.

\begin{definition}
\label{def:possible}
Let $X\subseteq \PP_\CC^{N-1}$ be a real projective variety with $\dim X=d$. 
We define the set of integers $\mathcal{N}(X)$ to be the possible numbers of real points that can be obtained after intersecting $X$ with a sufficiently general complementary dimension linear space. That is, 
$$
\mathcal{N}(X) := \left\{\, \#(X\cap W)_\RR:\,\, \begin{matrix} W \text{ real linear space with } \dim W=N-1-d \\  \text{ that intersects } X \text{ transversely} \end{matrix} \, \right\}.
$$
We call $\mathcal{N}(X)$ the {\em set of possible numbers of real solutions} for $X$.
\end{definition}

Our first contribution is to characterize $\mathcal{N}(X)$.
We denote the minimum  and maximum elements of the set $\mathcal{N}(X)$ by  $\mathcal{N}(X)_{\min}$ and $\mathcal{N}(X)_{\max}$. We assume that the variety has a smooth real point, to ensure its real locus is Zariski dense.

\begin{proposition}\label{thm: N(X)}
Let $X\subseteq \PP_\CC^{N-1}$ be a smooth real projective variety of dimension $d$ with a smooth real point. Then the set of possible numbers of real solutions $\mathcal{N}(X)$ satisfies
\begin{enumerate}[(i)]
    \item for $p\in \mathcal{N}(X)$, we have $p\equiv \deg X \mod 2$;
    \item $N-d \leq \mathcal{N}(X)_{\max} \leq \deg X$;
    \item $\mathcal{N}(X)=\{k: \mathcal{N}(X)_{\min} \leq k \leq \mathcal{N}(X)_{\max}, k\equiv \deg X \mod 2\}$.
\end{enumerate}
\end{proposition}

Proposition~\ref{thm: N(X)} may be known or intuitive to experts in real algebraic geometry, but to the best of our knowledge, a statement and proof is missing from the literature.
Our proof studies how real solutions change across the branch locus, which occurs in computing real homotopies between polynomial systems \cite{li1993solving}.
Here is an example to illustrate Proposition~\ref{thm: N(X)}.

\begin{example}
Consider the Edge quartic $C$ defined by 
\begin{equation}
    \label{eqn:edge}
25(x^4+y^4+z^4)-34(x^2y^2+x^2z^2+y^2z^2)=0,
\end{equation} taken from \cite{plaumann2011quartic}. 
It is one of the curves studied by William L. Edge in \cite{edge1938determinantal}, which admits a matrix representation over $\mathbb{Q}$.
Hyperplanes in $\PP_\CC^2$ 
can be viewed as points in $(\PP_\CC^2)^\ast$ with coordinates $[u,v,w]$. 
Generic hyperplanes
intersect $C$ transversely.
Those who intersect $C$ singularly form the dual curve $C^\vee$ defined by
\begin{align*}
10000u^{12} - 98600u^{10}v^2 - 98600u^{10}w^2 + 326225u^8v^4 + 85646u^8v^2w^2 + 326225u^8w^4 - 442850u^6v^6\\
-120462u^6v^4w^2 - 120462u^6v^2w^4 - 442850u^6w^6 + 326225u^4v^8 - 120462u^4v^6w^2 + 398634u^4v^4w^4\\
-120462u^4v^2w^6 + 326225u^4w^8 - 98600u^2v^{10} + 85646u^2v^8w^2 - 120462u^2v^6w^4 - 120462u^2v^4w^6\\
+85646u^2v^2w^8 - 98600u^2w^{10} + 10000v^{12} - 98600v^{10}w^2 
+ 326225v^8w^4 - 442850v^6w^6 \\
+ 326225v^4w^8 - 98600v^2w^{10} + 10000w^{12} = 0,
\end{align*}
see \cite[Example 5.2]{kaihnsa2019sixty}, where the authors study real lines that avoid $C$.
For any hyperplane in a fixed region of $(\PP_\RR^2)^\ast-C^\vee_\RR$, the number of real intersection points with $C$ is constant.
We plot $C^\vee_\RR$ and label each region of $(\PP^2_\RR)^\ast-C^\vee_\RR$ by the number of real intersection points with $C$ in Figure \ref{fig:dual curve}.
\begin{figure}[htbp]
    \centering
    \includegraphics[width=0.5\linewidth]{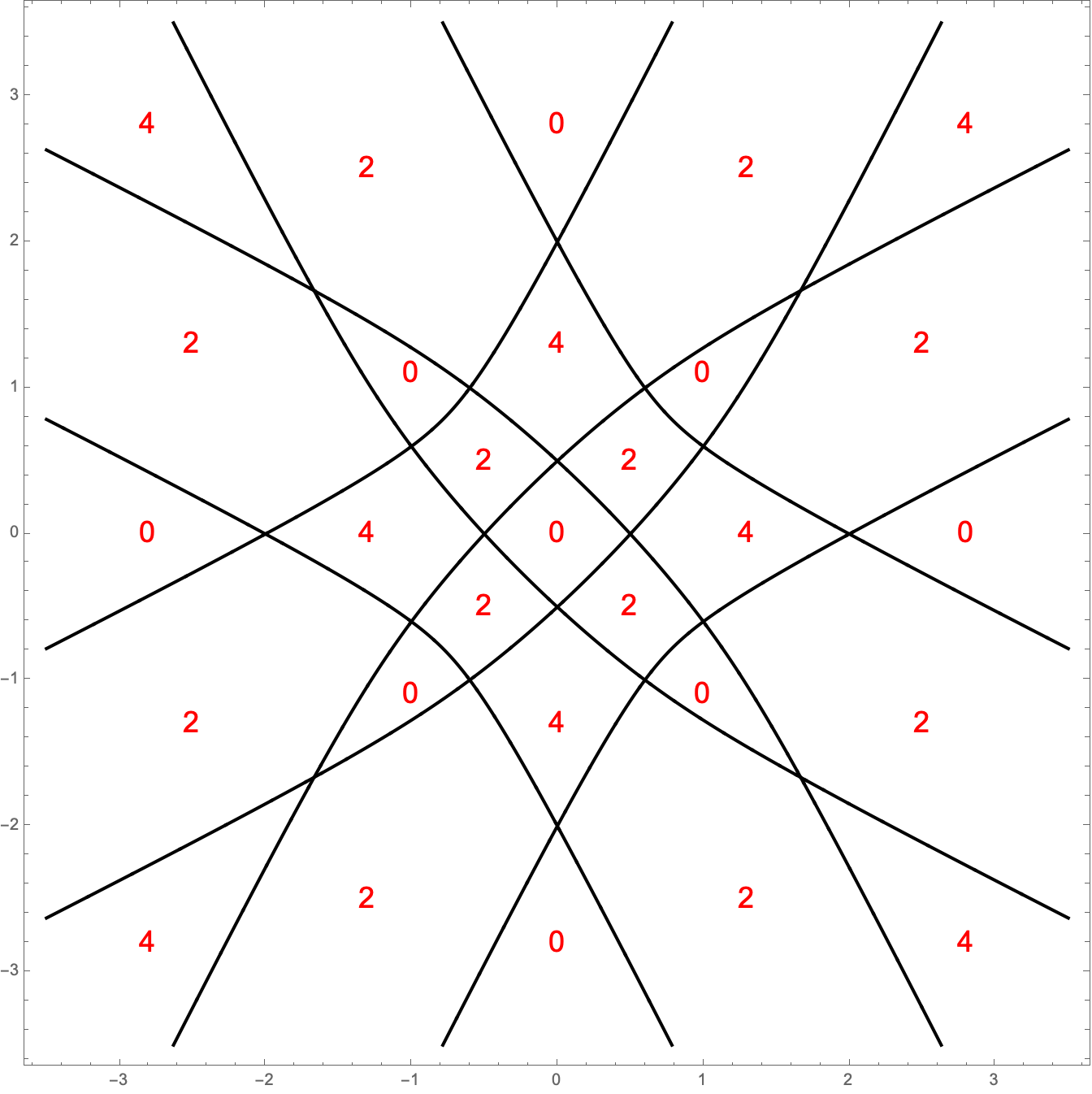}
    \caption{Dual curve of the Edge quartic~\eqref{eqn:edge} with regions labeled by the number of real intersection points
    }
    \label{fig:dual curve}
\end{figure}
If two regions are adjacent (that is, connected via smooth points in $C^\vee_\RR$) we see that their numbers of real intersection points differ by two.
\end{example}

We use Proposition \ref{thm: N(X)} to prove the following result.
Let $(X_\RR)^n$ denote the set of $n$-tuples of points on $X_\RR$ with the product topology. 
We say a probability measure on $(X_\RR)^n$ is \emph{strictly positive} if any non-empty open subset of $(X_\RR)^n$ has positive measure.

\begin{theorem}[A Real Generalized Trisecant Trichotomy]\label{thm: real trisecant lemma}
Let $X\subseteq \PP_\CC^{N-1}$ be a smooth real projective variety of dimension $d$ with a smooth real point.
Let $P_1,\ldots,P_n$ be points on $X_\RR$, sampled randomly from a strictly positive probability measure on $(X_\RR)^n$.
Let $W$ be the projective linear space they span. 
Then 
\begin{enumerate}[(a)]
    \item When $n+d<N$, $(X\cap W)_\RR=\{P_1,\ldots,P_n\}$ with probability $1$.
    \item  When $n+d=N$ and $\deg X\equiv n \mod 2$, 
    \begin{enumerate}[(i)]
        \item $(X\cap W)_\RR=\{P_1,\ldots,P_n\}$ with probability $0$ if $\mathcal{N}(X)_{\min}>n$;
        \item $(X\cap W)_\RR=\{P_1,\ldots,P_n\}$ with probability $0<p<1$ if $\mathcal{N}(X)_{\min}\leq n<\mathcal{N}(X)_{\max}$;
        \item $(X\cap W)_\RR=\{P_1,\ldots,P_n\}$ with probability $1$ if $\mathcal{N}(X)_{\max}=n$.
    \end{enumerate}
    \item When $n+d>N$ or $n+d=N$ and $\deg X\not\equiv n \mod 2$, $(X\cap W)_\RR = \{P_1,\ldots,P_n\}$ with probability~$0$. Moreover, when $n+d>N$, $(X\cap W)_\RR$ has positive dimension, so it contains infinitely many real points.
\end{enumerate}
\end{theorem}

The real trisecant trichotomy is studied for second Veronese embeddings in \cite{wang2024identifiability}; this will be discussed more in Section \ref{sec: applications}.
We now give examples to illustrate Theorem \ref{thm: real trisecant lemma}(b). 

\begin{example}[$p=0$]\label{ex: open question 1}
Let $X$ be the curve in $\PP_\CC^3$ of degree $k+2e+1$ defined as in Construction 1 of \cite{kummer2021huisman}, where $k,e\in \NN_{>0}$.
It has $\Ncal(X) = \{k-1,k+1\}$. For an even integer $k>4$, $\Ncal(X)_{\min}> \codim X + 1 = 3$. 
Hence, $\#(X\cap W)_\RR=3$ has probability $0$.

Another example follows from \cite[Theorem 4.6, Corollary 4.15]{kummer2022hyperbolic}. 
Let $X\subseteq \PP_\CC^{2k+1}$ be the projection of the rational normal curve in $\PP^n_\CC$ from the $(n-2k-2)$ dimensional linear space defined in \cite[Corollary 4.15]{kummer2022hyperbolic}, where $k$ can be any integer such that $2k+2\leq n$.
The degree of $X$ is $n$ by definition. 
By \cite[Theorem 4.6]{kummer2022hyperbolic}, a generic real hyperplane in $\PP_\CC^{2k+1}$ intersects $X$ in at most $2k$ complex points. Hence, $\Ncal(X)_{\min}\geq n-2k$. Note that $\codim X + 1 = 2k+1 < \Ncal(X)_{\min}$ whenever $n>4k+1$.

\end{example}

\begin{example}[$0<p<1$]
Let $X$ be the second Veronese embedding of $\PP_\CC^{I-1}$ in $\PP_\CC^{N-1}$ where $N={I+1 \choose 2}$ and suppose $I\equiv 2,3 \mod 4$. Then the probability that $\#(X\cap W)_\RR=N-I+1$ is in $(0,1)$ since $\Ncal(X) = \{0,2,\ldots,2^{I-1}\}$,
by the proof of \cite[Proposition 5.10]{wang2024identifiability}. 
\end{example}

\begin{example}[$p=1$]
Consider the plane curve $X$ defined by $x_0^4+x_1^4=x_2^4$ in $\PP_\CC^2$. Its real part does not intersect the line at infinity $x_2=0$ and is convex and simply closed.
A generic real line in the plane either intersects $X$ in two points or avoids $X$.
So, $\Ncal(X)=\{0,2\}$ and the probability that $\#(X\cap W)_\RR=2$
is 1, since $\Ncal(X)_{\max}$ is the codimension of $X$ plus one.
\end{example}

So far we have characterized the set $\mathcal{N}(X)$ in relation to its minimum and maximum elements, but we have not said what these minimum and maximum are. 
We now find the minimum and maximum elements of $\mathcal{N}(X)$ for special varieties of tensors.

Here we focus on the Segre-Veronese varieties.
They are varieties of rank one partially symmetric tensors, see e.g.~\cite{abo2013dimensions,raicu2012secant}. 
We denote the Segre-Veronese variety of $\PP_\CC^{m_1}\times \cdots \times \PP_\CC^{m_n}$ with degrees $d_1,\ldots,d_n$ by $\SV_{(m_1,\ldots,m_n)}(d_1,\ldots,d_n)$.
The varieties $\SV_{(m_1,\ldots,m_n)}(1,\ldots,1)$ are the usual Segre varieties.
Let $\mathbf{x}_i=[x_{i,0},\ldots,x_{i,{m_i}}]$ be the projective coordinates of $\PP_\CC^{m_i}$.
The Segre-Veronese variety~$\SV_{(m_1,\ldots,m_n)}(d_1,\ldots,d_n)$ is the image of the monomial map that sends $(\mathbf{x}_1,\ldots,\mathbf{x}_n)$ to $(\prod_{i=1}^n (x_{i,0})^{d_i},\ldots,\prod_{i=1}^n(x_{i,{m_i}})^{d_i})$, the vector of all monomials with multidegree $(d_1,\ldots,d_n)$. We denote the point on $\SV_{(m_1,\ldots,m_n)}(d_1,\ldots,d_n)$ corresponding to $(\mathbf{x}_1,\ldots,\mathbf{x}_n)$ by $\mathbf{x}_1^{\otimes d_1} \otimes \cdots \otimes \mathbf{x}_n^{\otimes d_n}$.
Considered in the affine cone over the projective space, it lies in the space of partially symmetric tensors $\Sym_{d_1}\RR^{m_1+1}\otimes \cdots \otimes \Sym_{d_n}\RR^{m_n+1}$.
We prove the following.

\begin{theorem}\label{thm: N(X) for segre-veronese} Let $X$ be the  Segre-Veronese variety $\SV_{(m_1,\ldots,m_n)}(d_1,\ldots,d_n)$. Then the set of possible numbers of real solutions $\mathcal{N}(X)$ satisfies
\begin{enumerate}[(i)]
    \item $\Ncal(X)_{\max}=\deg X =\frac{(m_1+\ldots+m_n)!}{m_1!\cdots m_n!} \prod_{i=1}^n d_i^{m_i}$;
    \item When at least two of $m_1,\ldots,m_n$ are odd, then $\Ncal(X)_{\min}=0$;
    \item When at least one of $d_1,\ldots,d_n$ is even, then $\Ncal(X)_{\min}=0$;
    \item When all $d_i$ are odd, $\Ncal(X)\supseteq \Ncal(\SV_{(m_1,\ldots,m_n)}(1,\ldots,1)).$
    \end{enumerate}
    \end{theorem}

We leave the remaining case as an open problem.
\begin{question}\label{question: segre_veronese}
What is $\Ncal(\SV_{(m_1,\ldots,m_n)}(d_1,\ldots,d_n))_{\min}$ when $d_1,\ldots,d_n$ are all odd and there is at most one odd integer among $m_1,\ldots,m_n$?
\end{question}

We investigate the set $\Ncal(X)$ for small Segre varieties that fall under the setting of Question \ref{question: segre_veronese}. For each Segre variety considered, we sample random polynomial systems.
We use numerical homotopy computation methods from \cite{breiding2018homotopycontinuation} to compute the number of real solutions for each system, see Table \ref{table: segre}. 
Theorem~\ref{thm: N(X) for segre-veronese}(i) says that $\deg(X)$ real solutions occurs with positive probability.
However, in all but the first row, degree many real solutions did not occur in our finite samples, suggesting that its probability is small. 
For results on real root counts of random polynomials,
see \cite{dembo2002random,schehr2008real}.

\begin{table}[htbp]
\centering
\label{table: segre}
\begin{tabular}{|c|c|c|c|}
\hline
Segre of & degree & integer & Gaussian\\
\hline
$\mathbb{P}^2 \times \mathbb{P}^2$& 6& $[0,6]_2$& $[0,6]_2$\\
\hline 
$\PP^2 \times \PP^4$& 15 & $[1,13]_2$ & $[1,13]_2$\\
\hline
$\PP^2 \times \PP^2 \times \PP^2$& 90 & $[2,32]_2$ & $[4,30]_2$\\
\hline
$\PP^2 \times \PP^6$& 28& $[0,18]_2$ & $[0,18]_2$\\
\hline
$\PP^4 \times \PP^4$& 70& $[2,28]_2$ & $[2,26]_2$\\
\hline
$\PP^2 \times \PP^2 \times \PP^4$ & 420 & $[12,60]_2$ & $[14,62]_2$\\
\hline
$\PP^2 \times \PP^2 \times \PP^2 \times \PP^2$ & 2520 & $[70,146]_2$ & $[68,146]_2$\\
\hline
\end{tabular}
\caption{The number of real solutions obtained for different Segre varieties. 
We generate coefficients in two ways: random integer values in the range $[-20,20]$ and sampling from a standard Gaussian. 
We record the possible numbers of real solutions we obtain over 10000 sampled systems.
We use $[m,n]_2$ to denote all integers with the same parity as $m,n$ in the interval $[m,n]$. In all cases, the possible number of real solutions obtained is of this form and the frequencies (not displayed) are unimodal (e.g. for $\PP^2 \times \PP^2$, the frequencies for 0, 2, 4, 6 are 469, 5219, 3603, 709 respectively).
}
\end{table}

In special cases, we can say more about the set $\Ncal(X)$.

\begin{theorem}\label{thm: N(X) for small segre-veronese}
\begin{enumerate}[(i)]
\item $\Ncal(\SV_{(1,n)}(1,1))= \{\, k\,: 0 \leq k\leq n+1,\,\, k\equiv n+1 \mod 2 \, \}$;
\item $\Ncal(\SV_{(2,n)}(1,1))\supseteq \{\, k\,: \lfloor \frac{n-2}{2} \rfloor \leq k\leq \deg \SV_{(2,n)}(1,1),\,\, k\equiv  \deg \SV_{(2,n)}(1,1)\mod 2  \,\},$ for $n\geq 2$.
\end{enumerate}
\end{theorem}

Another interesting family of varieties are those with $\Ncal(X)_{\max}=n$ for a $d$-dimensional variety $X\subseteq \PP_\CC^{N-1}$ with $n=N-d$. We say these varieties have $\Ncal(X)_{\max}$ minimal.
These varieties have probability 1 of $(X\cap W)_\RR=\{P_1,\ldots,P_n\}$ in Theorem \ref{thm: real trisecant lemma}$(b)$.
We characterize the plane curves $X$ with $\Ncal(X)_{\max}$ minimal. We also construct hypersurfaces $X$ with $\Ncal(X)_{\max}$ minimal for any dimension and even degree.

The rest of the paper is organized as follows.
We prove Proposition \ref{thm: N(X)}
in Section~\ref{sec:N(X)}. 
We prove Theorem \ref{thm: real trisecant lemma} in 
Section \ref{sec:real trisecant trichotomy}.
We prove Theorem \ref{thm: N(X) for segre-veronese} and Theorem \ref{thm: N(X) for small segre-veronese} in Section \ref{subsec: segre-veronese}.
We construct varieties with $\Ncal(X)_{\max}$ minimal
in Section \ref{subsec: N(X) max minimal}. We explore the applications of the real generalized trisecant trichotomy to independent component analysis, tensor decompositions, and the study of typical tensor ranks in Section \ref{sec: applications}.

\section{The possible numbers of real solutions}\label{sec:N(X)}
In this section, we prove Proposition \ref{thm: N(X)}, which studies 
 the possible numbers of real points that can be obtained after intersecting a variety with a sufficiently general complementary dimension linear space. 
Throughout this section, $X\subseteq \PP_\CC^{N-1}$ is a smooth real projective variety of dimension $d$ with a smooth real point.

\begin{proof}[Proof of Proposition \ref{thm: N(X)}(i,ii) and the $\subseteq$ part of (iii)]
Let $W$ be a real linear space of dimension $N-1-d$ that intersects $X$ transversely. 
The intersection $X\cap W$ is the vanishing locus of real polynomials so complex points appear in pairs. 
It contains $\deg X$ many points so $\#(X\cap W)_\RR \equiv \deg X \mod 2$. 
This proves (i). We also obtain, for (iii), that 
$$\Ncal(X)\subseteq \{\,k\,: \mathcal{N}(X)_{\min} \leq k \leq \mathcal{N}(X)_{\max},\,\, k\equiv \deg X \mod 2\,\}.$$
For (ii), the inequality $\Ncal(X)_{\max}\leq \deg X$ holds, since this is the number of complex intersection points.
For the inequality $\Ncal(X)_{\max} \geq N - d$, 
we construct a sufficiently general linear space $W$ of complementary dimension to $X$ that intersects $X$ in at least $N-d$ points.
Let $p$ be a real smooth point of $X$. 
Then the local dimension of $X_\RR$ at $p$ is $\dim X=d$ by \cite[Proposition 7.6.2]{bochnak2013real}, in other words, there is a semi-algebraic neighborhood $U$ of $p$ in $X_\RR$ of dimension equal to $\dim X=d$. 
The variety $X$ is non-degenerate, so $N-d$ generic points in $U$ are linearly independent. 
We denote the linear space they span by $W$. It has complementary dimension to $X$ in $\PP_\CC^{N-1}$ and $(X\cap W)_\RR$ contains at least $N-d$ points, since $W$ is generated by $N-d$ points in $X_\RR$.
\end{proof}

To prove Proposition \ref{thm: N(X)}, it remains to show that $$\Ncal(X)\supseteq \{\,k\,: \mathcal{N}(X)_{\min} \leq k \leq \mathcal{N}(X)_{\max},\,\, k\equiv \deg X \mod 2\,\}.$$
Our proof uses the following definition. 

\begin{definition}\label{def: U_k,Z_X}
We define $\Ucal_k\subseteq \Gr(N-d-1,N-1)_\RR$ to be the set of $(N-d-1)$-dimensional linear spaces in $\PP_\RR^{N-1}$ that intersect $X$ transversely in exactly $k$ real intersection points.
\end{definition}

We prove that 
the union of disjoint open chambers $\bigcup_{k\in \Ncal(X)} \Ucal_k$ 
is a dense open set in the Grassmannian $\Gr(N-d-1,N-1)_\RR$. Note that a set $\Ucal_k$ can be disconnected.
For two points, one generic in $\Ucal_{\Ncal(X)_{\min}}$ and the other generic in $\Ucal_{\Ncal(X)_{\max}}$, we show that we can travel from one to the other via a continuous path such that 
each time we travel from one chamber to another, we go from some $\Ucal_k$ to $\Ucal_{k+2}$ or to $\Ucal_{k-2}$. 
As we start in $\Ucal_{\Ncal(X)_{\min}}$ and end in $\Ucal_{\Ncal(X)_{\max}}$, every set $\Ucal_k$ for $k$ an integer with the same parity as $\deg X$ in the interval $[\Ncal(X)_{\min}, \Ncal(X)_{\max}]$ will be visited. 
See Figure \ref{fig:dual curve} for an example and also for the disconnectedness of the $\Ucal_k$.
We start by studying the topology of the sets $\Ucal_k$ in $\Gr(N-d-1,N-1)_\RR$.

\begin{lemma}\label{lem: topology for Grassmannian chambers}
\begin{enumerate}[(a)]
\item For $k\in \Ncal(X)$, the set $\Ucal_k$ is non-empty and open in $\Gr(N-d-1,N-1)_\RR$ with the Euclidean topology.
\item The set  $B:=\Gr(N-d-1,N-1)_\RR-\bigcup_{k\in \Ncal(X)} \Ucal_k$ is a hypersurface. It is the boundary of $\bigcup_{k\in \Ncal(X)} \Ucal_k$ and  contains linear spaces in $\PP_\RR^{N-1}$ that intersect $X$ at some point with multiplicity at least two or in some positive dimension variety.
\end{enumerate}
\end{lemma}
\begin{proof}
Let $V$ be a codimension $d$ linear space in $\PP_\CC^{N-1}$ that intersects $X$ transversely in exactly $k$ real solutions.
Roots of a polynomial system change continuously as its coefficients change. 
So if we perturb $V$ in a small open neighborhood around it, all solutions to $V\cap X$ are distinct and the complex points in $V\cap X$ move to complex points.
By considering the Grobner bases of the ideal generated by polynomials defining $X$ and the linear relations defining $V$, the $i$-th coordinate of the solutions to $V\cap X$ are roots of some univariate polynomial with coefficients that change continuously as we move $V$. 
So, the real points in $V\cap X$ remain real.
Hence, there is an open neighborhood of $V$ in $\Ucal_k$, thus $\Ucal_k$ is open.

When we leave $\Ucal_k$ and enter $\Ucal_{k'}$ along some path in $\Gr(N-d-1,N-1)_\RR$, we must have at least two solutions coming together on the boundary of $\Ucal_k$ and $\Ucal_{k'}$. So, the set $$B:=\Gr(N-d-1,N-1)_\RR-\bigcup_{k\in \Ncal(X)}\Ucal_k$$ is the collection of linear spaces in $\PP_\RR^{N-1}$ of codimension $d$ that intersect $X$ singularly.
The set $B$ is the vanishing locus of the Hurwitz form 
for $X$ and it is an irreducible hypersurface since $X$ is irreducible, see \cite[Theorem 1.1]{sturmfels2017hurwitz}.
\end{proof}

We will travel from $\Ucal_{\Ncal(X)_{\min}}$ to $\Ucal_{\Ncal(X)_{\max}}$ via a connected sequence of lines in $\Gr(N-d-1,N-1)_\RR$. 
Lines in $\Gr(N-d-1,N-1)$ are pencils of linear spaces in $\PP_\CC^{N-1}$.
We use lines to form our path and will make the lines sufficiently generic to reduce to the case where $X$ is an algebraic curve.
We can then use the notion of dual varieties to understand hyperplanes that intersect an algebraic curve non-transversely.

\begin{lemma}
A line in $\Gr(N-d-1,N-1)$ is a pencil of linear spaces of dimension $N-d-1$ that contain a fixed $(N-d-2)$-space and are contained in a fixed $(N-d)$-space in $\PP_\CC^{N-1}$.
\end{lemma}
\begin{proof}
This is the description of a line in the Grassmannian as a Schubert cycle, which is well-known.  We include a proof for convenience.
Let $L$ be a line in $\Gr(N-d-1,N-1)$. 
Let $V_L= \bigcup_{[V]\in L} V\subseteq \PP_\CC^{N-1}$.  Since a line in $\Gr(N-d-1,N-d)$ is the pencil of $(N-d-1)$-spaces that contain a $(N-d-2)$-space, it is enough to show that $V_L$ is a $(N-d)$-space, i.e. a projective variety of dimension $N-d$ and degree one in the Pl\"ucker embedding of $\Gr(N-d-1,N-1)$.
Let $$\Phi= \{\,(V,p)\in \Gr(N-d-1,N-1)\times \PP_\CC^{N-1}:\, [V]\in \Gr(N-d-1,N-1),\,\, p\in V \,\}$$ and let $p_1$ and $p_2$ be the projection of $\Phi$ onto its first and second factor, respectively.
Then $V_L= p_2(p_1^{-1}(L))$ is a projective algebraic variety with dimension $N-d$ and it is irreducible since $p_1^{-1}(L)$ is irreducible. 
Let $W$ be a generic dimension $d$ linear space in $\PP_\CC^{N-1}$, then 
$$\deg (V_L)= \#( V_L\cap W ) = 
\# \{ V\in L: V\cap W \neq \emptyset \}.$$ 
Let $\Sigma$ be the set of points in $\Gr(N-d-1,N-1)$ whose corresponding linear spaces intersect $W$. It is a hyperplane section of $\Gr(N-d-1,N-1)$ in the Pl\"ucker embedding.
Hence $$\# \{ V\in L: V\cap W\neq \emptyset \}= \# (L\cap \Sigma)= 1$$ and $V_L$ has degree one.
\end{proof}

We recall the definition of a dual variety and some of its key properties.

\begin{definition}
Let $X\subseteq \PP_\CC^{N-1}$ be an irreducible smooth algebraic variety. 
The {\em dual variety} $X^\vee\subseteq (\PP_\CC^{N-1})^\ast$ is the collection of hyperplanes that intersect $X$ singularly.
In general, $X^\vee$ is a hypersurface, otherwise $X$ is ruled by projective spaces of dimension $\codim X^\vee-1$.
\end{definition}

\begin{proposition}\label{prop: bad locus small}
Let $Y\subseteq \PP_\CC^{N-1}$ be a non-degenerate smooth algebraic curve with $Y^\vee$ a hypersurface in $(\PP_\CC^{N-1})^\ast$. 
If $H$ is a hyperplane that intersects $X$ in at least two points with multiplicity two or at least one point with multiplicity at least three (including positive dimensional intersection), then $[H] \in \Sing(Y^\vee)$.
\end{proposition}
\begin{proof}
By a result of \cite{dimca1986milnor} (for a proof see \cite[Theorem 10.8]{tevelev2006projective}), the multiplicity of $Y^\vee$ at $[H]$ is equal to $\Sigma_{p\in \Sing(Y\cap H)}\, \mu(Y\cap H,p)$ where $\mu(Y\cap H,p)$ is the Milnor number. 
If $p$ is an intersection point with multiplicity $m$, then $\mu(Y\cap H,p)=m-1$. 
So, if $H$ intersects $Y$ in at least two points with multiplicity two or at least one point with multiplicity at least three (including positive dimensional intersection), we obtain 
$\Sigma_{p\in \Sing(Y\cap H)}\, \mu(Y\cap H,p)\geq 2$, so $[H]$ is a singular point in $Y^\vee$.
\end{proof}

\begin{lemma}\label{lem: generic line change by 2 each time}
Let $Y \subseteq \PP_\CC^{N-1}$ be a non-degenerate smooth algebraic curve defined by real polynomials with a smooth real point. 
A generic real line $L$ in $(\PP_\CC^{N-1})^\ast$ will intersect $Y^\vee$ transversely. 
Moreover, each time the line crosses $Y^\vee$, it travels from $\Ucal_k$ to $\Ucal_{k+2}$ or to $\Ucal_{k-2}$ for some $k\in \Ncal(Y)$.
\end{lemma}
\begin{proof}
If $Y^\vee$ is a hypersurface in $(\PP_\CC^{N-1})^\ast$, $\Sing(Y^\vee)$ has codimension at least two in $(\PP_\CC^{N-1})^\ast$. 
If $Y^\vee$ is not a hypersurface in $(\PP_\CC^{N-1})^\ast$, 
then $\Sing(Y^\vee)$ has codimension at least three. 
A generic real line will intersect $Y^\vee$ transversely in both cases.

Any hyperplane corresponding to a point in $L$ intersects $Y$ in at most one point with multiplicity two, by Proposition \ref{prop: bad locus small}. 
So, each time $L$ crosses $Y^\vee$, two distinct points become one double point and then become two distinct points.
Since $L$ is real, the double point must be real otherwise there would be a pair of complex conjugate double points and we would have a bitangent. 

Suppose $L$ crosses $Y^\vee$ at some hyperplane $H_0$ where $Y\cap H_0$ has a double point $p$ and consider $H_t$ in $L$ for $t$ close to 0.  
First notice that the tangent line to $Y$ at $p$ lies in $H_0$.   
Since $Y$ is smooth at $p$, $H_t$ intersects $Y$ at some point $p_t$ close to $p$.  
There are two possibilities:
\begin{enumerate}[(i)]
    \item if $p_t$ is not real, its conjugate is also close to $p$. The secant line from $p_t$ to its conjugate is a real line in $H_t$. This secant line moves to the tangent line at $p$ as $t$ moves to $0$;
    \item if $p_t$ is real there is a real secant to $Y$ through $p_t$ that is contained in $H_t$ and moves to the tangent line at $p$ as $t$ moves to $0$.  
\end{enumerate}

Consider possibilities (i) and (ii) as $t$ changes sign at 0.  
If (i) occurs on both sides, $p$ is an acnode on $Y$, an isolated real singularity, against the smoothness assumption. 
If (ii) occurs on both sides, $Y$ has a real node at $p$, again a singularity against the assumption.  
So (i) and (ii) occur each on one side.
Since all other real intersections of $H_0$ with $Y$ are transversal, i.e. at smooth points, these real intersections will remain real and distinct for $t$ close to $0$. 
So the number of real intersection points of $H_t$ with $Y$ changes by two as $t$ passes by $0$.
\end{proof}
We use Bertini's Theorem, \cite[Theorem 17.16]{harris2013algebraic}, to reduce our problem to the smooth curve case in order to apply Lemma \ref{lem: generic line change by 2 each time}.
It states that for a smooth variety, a generic member of a linear system of divisors on X is smooth away from the base locus of the system. 
It extends to the following.

\begin{corollary}[of Bertini's Theorem] 
Let $X$ be a smooth variety. 
If the base locus of a linear system of divisors on $X$ is empty or a finite reduced set of points, then the general member of the linear system is smooth. 
\end{corollary}
\begin{proof} 
It suffices to check for singularities at the basepoints.  If every divisor is singular at a base point $p\in X$, i.e. has multiplicity at least two at $p$, then the base locus contains the first order neighborhood of $p$.  This is against the hypothesis that the base locus is a reduced set of points, so the theorem holds.
\end{proof}

When we refer to Bertini's theorem, we include this corollary.
By a real line in $\Gr(N-d-1,N-1)_\RR$, we mean a line corresponding to a pencil of real $(N-d-1)$-spaces containing a real $(N-d-2)$-space in a real $(N-d)$ space.
We show that we can construct a sequence of real lines in $\Gr(N-d-1,N-1)_\RR$ connecting a point in $\Ucal_{\Ncal(X)_{\min}}$ to a point in $\Ucal_{\Ncal(X)_{\max}}$, where the $(N-d)$ space corresponding to each line intersects $X$ in a smooth curve.

For two general codimension $d$ linear spaces $W_1,W_2$ in $\PP_\CC^{N-1}$, the intersection is empty if $2d>N-1$ and has codimension $2d$ if $2d\leq N-1$.
We define $\ell$ to be $-1$ if $2d>N-1$ and $N-1-2d$ if $2d \leq N-1$. 
It is the dimension of $W_1\cap W_2$ ($-1$ when the intersection is empty).
Let $k=N-d-\ell-1$.
The span, $\Span(W_1,W_2)$, has dimension $2N-2-2d-\ell$, so each $W_i$ has codimension $k$ in this span.
This is also the codimension of $W_1\cap W_2$ in $W_i$ for $i=1,2$.

\begin{lemma}\label{lem: can travel via lines in Grassmannian} 
Fix $X\subseteq \PP_\CC^{N-1}$ of dimension $d$.
Suppose $V_0,V_k$ are two generic codimension $d$ real linear spaces in $\PP_\CC^{N-1}$. 
Then $(V_0\cap V_k)\cap X=\emptyset$, $\Span(V_0,V_k)\cap X$ is smooth with dimension $k$,
and there are real linear spaces $V_1,\ldots,V_{k-1}$ of codimension $d$ and real linear spaces $U_0,\ldots,U_{k-1}$ in $\Span(V_0,V_k)$ of codimension $d-1$
satisfying the following properties:
\begin{enumerate}[(i)]
    \item $U_i=\Span(V_i,V_{i+1})$ and $\codim_{V_i}(V_i\cap V_{i+1})=1$ for $i=0,\ldots,k-1$;
    \item $V_{i+1} \cap V_k\supseteq V_i\cap V_k$ with $\dim({V_{i+1} \cap V_k})=\dim({V_i \cap V_k})+1$ and $V_{i+1}\cap V_0\subseteq V_i\cap V_0$ with $\dim({V_{i+1} \cap V_0})=\dim({V_i \cap V_0})-1$ for $i=0,\ldots,k-1$;
    \item $(V_i\cap V_k)\cap X=\emptyset$ and $\Span(V_i,V_k)\cap X$ is smooth with dimension $k-i$ for $i=0,\ldots,k$;
    \item $V_i$ intersects $X$ transversely for $i=1,\ldots,k-1$ and $U_i\cap X$ is a smooth curve for $i=0,\ldots,k-1$.
\end{enumerate}
\end{lemma}

\begin{proof}
For generic linear spaces $V_0$ and $V_k$, the linear space $V_0\cap V_k$ is generic with dimension $\ell$ or is empty and $\Span(V_0,V_k)$ is generic with dimension $N-d-1+k$. 
So, $(V_0\cap V_k)\cap X=\emptyset$ and $\Span(V_0,V_k)\cap X$ is smooth with dimension $k$ by the smoothness of $X$.
When $k=1$, there is nothing to prove.

When $k=2$, we choose $U_0$ to be a generic linear space of dimension $N-d$ in $\Span(V_0,V_2)$ containing $V_0$. 
By Bertini's Theorem, $U_0\cap X$ is a smooth curve since $\Span(V_0,V_2)\cap X$ is smooth and the base locus $V_0\cap X$ has dimension 0.
We choose $U_1$ to be a linear space of dimension $N-d$ in $\Span(V_0,V_2)$ containing $V_2$ such that, again 
by Bertini's Theorem, the intersections $U_1\cap X$  and $U_0\cap U_1\cap X$ are both smooth.
We define $V_1=U_0\cap U_1$.  
Then $U_1\cap X$ is a smooth curve and $V_1\cap X$ is smooth and finite. 

Now suppose $k>2$. 
We construct $V_1,\ldots,V_{k-1}$ and $U_0,\ldots,U_{k-1}$ inductively.
Suppose we have already constructed $V_1,\ldots,V_i$ and $U_0,\ldots,U_{i-1}$.
Suppose first that $k-i>2$. 
Note that $\Span(V_i,V_k)\cap X$ smooth and $(V_i\cap V_k)\cap X=\emptyset$.
We choose $U_i$ to be a generic linear space of dimension $N-d$ in $\Span(V_i,V_k)$ containing $V_i$ such that by Bertini's Theorem, $U_i\cap X$ is smooth and $(U_i\cap V_k)\cap X$ is empty.
We define $V_{i+1}$ to be a generic $N-d-1$ dimensional linear space in $U_i$ containing $U_i\cap V_k$ such that by Bertini's Theorem, $V_{i+1}\cap X$ is transverse and $\Span(V_{i+1}, V_k)\cap X$ is smooth.
Note that we have $U_i\cap V_k= V_{i+1}\cap V_k$ because $\codim_{V_{i+1}}(V_i\cap V_{i+1})=1$, $U_i\cap V_k$ contains $V_i\cap V_k$ and $\dim(U_i\cap V_k)=\dim (V_i\cap V_k)+1$.
So, in particular $(V_{i+1}\cap V_k)\cap X=\emptyset$. 
If $k-i=2$, we use the same argument as $k=2$ above.
\end{proof} 

Now, we prove the $\supseteq$ half of Proposition \ref{thm: N(X)}(iii); i.e., we show that $$\Ncal(X)\supseteq \{\,k:\, \mathcal{N}(X)_{\min} \leq k \leq \mathcal{N}(X)_{\max},\,\, k\equiv \deg X \mod 2\,\}.$$

\begin{proof}[Proof of the $\supseteq$ half of Proposition \ref{thm: N(X)}(iii)]

By Lemma \ref{lem: topology for Grassmannian chambers}(a), there are nonempty open sets $\Ucal_j\subseteq \Gr(N-d-1,N-1)_\RR$ for all $j\in \Ncal(X)$ such that any element in $\Ucal_j$ intersects $X$ in precisely $j$ real distinct points.
We take $V_0, V_k$ to be generic points in $\Ucal_{\Ncal(X)_{\min}},\Ucal_{\Ncal(X)_{\max}}$ respectively. 
By Lemma \ref{lem: can travel via lines in Grassmannian}, there are $N-d-1$ dimensional linear spaces $V_1,\ldots,V_{k-1}$ that intersect $X$ transversely and $N-d$ dimensional linear spaces $U_0,\ldots,U_{k-1}$ that intersect $X$ in a smooth curve with $V_{i},V_{i+1}\subseteq U_i$ for $i=0,\ldots, k-1$.
It is enough to show that inside each $U_i$, we can travel from $V_i$ to $V_{i+1}$ via a continuous path in $\Gr(N-d-1,N-1)_\RR$ and each time we leave an open chamber $\Ucal_k$, we enter either $\Ucal_{k-2}$ or $\Ucal_{k+2}$.
We denote $X_i= X\cap U_i$ and we treat $U_i\cong \PP_\CC^{N-d}$ as our ambient space. 
Now, $V_i,V_{i+1}$ are real points in $(\PP_\CC^{N-d})^\ast$.
Since $V_i,V_{i+1}$ intersect $X$ transversely, there are open neighborhoods $D_i,D_{i+1}\subset (\PP_\RR^{N-d})^\ast$ around $[V_i]$ and $[V_{i+1}]$ respectively, such that $D_j$ is contained in the connected component of $[V_j]$ in $(\PP_\RR^{N-d})^\ast - (X_i)^\vee$ for $j=\{i,i+1\}$.
We pick generic points $W_i\in D_i, W_{i+1}\in D_{i+1}$, then the line segments $(V_i,W_i), (V_{i+1}, W_{i+1})$ don't cross $(X_i)^\vee$.
The line $(W_i,W_{i+1})$ is a general line in $(\PP^{N-d})^\ast$, i.e. has transverse intersection with $Y^\vee$. If $X_i$ has a real smooth point, by Lemma \ref{lem: generic line change by 2 each time}, each time the line crosses $Y^\vee$, it travels from some $\Ucal_k$ to $\Ucal_{k+2}$ or $\Ucal_{k-2}$ for some $k\in \Ncal(Y)$.
If $X_i$ does not have a real smooth point, then $(X_i)_\RR$ consists of a finite number of singular points and a generic line does not intersect it.
So, the lines connecting $W_i$ to $W_{i+1}$ stay in $\Ucal_0$.
\end{proof}

We obtain the following corollary from the proof of Proposition \ref{thm: N(X)}(iii).

\begin{definition} 
Let $\Zcal_X\subseteq \Gr(N-d-1,N-1)_\RR$ be the set of $(N-d-1)$-spaces in $\PP_\RR^{N-1}$ that intersect $X$ in at least two points with multiplicity two or one point with multiplicity at least three (including positive dimensional intersection).
\end{definition}
 
\begin{corollary}
The set $\Gr(N-d-1,N-1)_\RR-\Zcal_X$ is path-connected. 
If $\Ucal_i,\Ucal_j$ are smoothly adjacent, meaning that $\overline{\Ucal_i}\cap \overline{\Ucal_j}$ contains some smooth point of the boundary $B:=\Gr(N-d-1,N-1)_\RR-\bigcup_{k\in \Ncal(X)} \Ucal_k$, then $i=j+2$ or $i=j-2$.
\end{corollary}

\section{Real trisecant trichotomy}\label{sec:real trisecant trichotomy}
In this section, we prove Theorem \ref{thm: real trisecant lemma}. 
Throughout the section, we suppose $X\subseteq \PP_\CC^{N-1}$ is a smooth real projective variety of dimension $d$ with a smooth real point  
and that $P_1,\ldots,P_n$ are points on $X_\RR$, sampled randomly from a strictly positive probability measure on $(X_\RR)^n$.

\begin{proposition}
Let $P_1,\ldots,P_n$ be points on $X_\RR$ sampled randomly from a strictly positive probability measure on $(X_\RR)^n$.
When $n+d<N$, $(X\cap W)_\RR=\{P_1,\ldots,P_n\}$ with probability 1.
\end{proposition}
\begin{proof}
It suffices to show that for general points $P_1,\ldots,P_n\in X_\RR$, we have $(X\cap W)_\RR=\{P_1,\ldots,P_n\}$.
General real points are also general complex points since $X$ has a smooth real point and $X_\RR$ is Zariski-dense in $X$. 
By the generalized trisecant lemma (see Theorem~\ref{lemma:trisecant}) we have $X\cap W=\{P_1,\ldots,P_n\}$ and hence
$(X\cap W)_\RR=\{P_1,\ldots,P_n\}$.
\end{proof}

\begin{lemma}
Let $W$ be a generic linear space with complementary dimension to $X$ in $\PP_\CC^{N-1}$. 
Then any subset of $\dim W+1$ intersection points is linearly independent.
\end{lemma}
\begin{proof}
Denote $\dim W=n-1$. The intersection $X\cap W$ is non-degenerate, by \cite[Proposition 18.10]{harris2013algebraic}, so $W$ is spanned by a subset of $n$ intersection points.
Assume for contradiction that, for a generic linear space $W$ of dimension $n-1$, there is a linearly dependent subset of $n$ intersection points. 
This linearly dependent set of intersection points spans a linear space 
$V_W$ of dimension at most $k\leq n-2$.
Let $Y_k$ be the collection of $k$ dimensional linear spaces in $\PP^{N-1}$ spanned by $k+1$ points in $X$ such that its intersection with $X$ contains more than these $k+1$ points.
By the generalized trisecant lemma, $Y_k$ has positive codimension in the image of the map $\phi: X^{k+1} \dashrightarrow \Gr(k,N-1)$ which sends $k+1$ points on $X$ to the linear space they span.
So, $\dim Y_k< (k+1)\dim X$.
By the hypothesis, a generic linear space in $\Gr(n-1,N-1)$ is spanned by a linear space in $Y_k$ for some $k\leq n-2$ and $n-1-k$ other points in $X$, so $$\dim\Gr(n-1,N-1)\leq \max_{k\leq n-2}\{\dim Y_k + (n-1-k) \dim X\}.$$
But we have $\dim Y_k + (n-1-k) \dim X< (k+1)\dim X+ (n-1-k) \dim X = (N-n)n =\Gr(n-1, N-1)$, a contradiction.
\end{proof}

\begin{proposition}\label{prop: complementary dimension case}
Let $P_1,\ldots,P_n$ be points on $X_\RR$, sampled randomly from a strictly positive probability measure on $(X_\RR)^n$. Define $W=\Span\{P_1,\ldots,P_n\}$. 
Let $n=N-d$ and we assume that $\deg X\equiv n \mod 2$. Then,
    \begin{enumerate}[(i)]
        \item $(X\cap W)_\RR=\{P_1,\ldots,P_n\}$ with probability $0$ if $\mathcal{N}(X)_{\min}>n$;
        \item $(X\cap W)_\RR=\{P_1,\ldots,P_n\}$ with probability $0<p<1$ if $\mathcal{N}(X)_{\min}\leq n<\mathcal{N}(X)_{\max}$;
        \item $(X\cap W)_\RR=\{P_1,\ldots,P_n\}$ with probability $1$ if $\mathcal{N}(X)_{\max}=n$.
    \end{enumerate}
\end{proposition}
\begin{proof}
Consider the map $\phi: (X_\RR)^n \dashrightarrow \Gr(n-1,N-1)_\RR$ that sends $n$ points on $X_\RR$ to the $n-1$ dimensional linear space they span whenever the $n$ points are linearly independent. The map $\phi$ is continuous. 

Recall from Definition \ref{def: U_k,Z_X} and Lemma \ref{lem: topology for Grassmannian chambers} that for each $k\in \Ncal(X)$, there is a nonempty open set $\Ucal_k\subseteq \Gr(n-1,N-1)_\RR$ parameterizing real $(n-1)$-spaces that intersect $X$ transversely in exactly $k$ real points. 
We will show that for $k\in \Ncal(X)$ and $k\geq n$ (if such a $k$ exists), the set $\phi^{-1}(\Ucal_k)$ is nonempty and it is open since $\phi$ is continuous.
For a generic linear space $W\in \Ucal_k$, any subset of $\dim V+1=n$ intersection points of $W\cap X$ is linearly independent. 
The intersection $W\cap X$ has $k$ real intersection points with $k\geq n$, so we can choose $n$ real intersection points to span $W$. 
In particular, $W$ is in the image of $\phi$ and $\phi^{-1}(\Ucal_k)$ is non-empty and open.
Moreover, the closure of $\bigcup_{k\in \Ncal(X), k\geq n} \phi^{-1}(\Ucal_k)$ is the domain of $\phi$.

Note that $(X\cap W)_\RR = \{P_1,\ldots, P_n\}$ if and only if $\{P_1,\ldots,P_n\}\in \phi^{-1}(\Ucal_n)$. So, the probability is 1 when $\phi^{-1}(\Ucal_n)$ is dense. 
In this case, if $k\in \Ncal(X)$ and $k\geq n$, we must have $k=n$, so $\Ncal(X)_{\max}=n$.
The probability is in the interval $(0,1)$ when $\phi^{-1}(\Ucal_n)$ is non-empty and open but not dense. 
This happens when $\mathcal{N}(X)_{\min}\leq n<\mathcal{N}(X)_{\max}$.
Finally, the probability is $0$ when $\phi^{-1}(\Ucal_n)$ is empty; i.e., when $\Ncal_{\min}>n$.
\end{proof}

\begin{proposition}
Let $P_1,\ldots,P_n$ be random points on $X$ that follow a strictly positive probability measure on $(X_\RR)^n$. Let $W=\Span\{P_1,\ldots,P_n\}$. 
Assume $n+d=N$ and $\deg X\not\equiv n \mod 2$ or $n+d>N$.
Then $$(X\cap W)_\RR \supsetneqq \{P_1,\ldots,P_n\}.$$ 
Moreover, when $n+d>N$, $(X\cap W)_\RR$ contains infinitely many real points.   
\end{proposition}
\begin{proof}
It suffices to show the result for general $P_1,\ldots,P_n\in X_\RR$ for which $W$ intersects $X$ smoothly at the points $P_1,\ldots, P_n$, since such $(P_1,\ldots,P_n)$ occurs almost surely.
Suppose first $n+d=N$ and $\deg X\not\equiv n \mod 2$. The intersection $X\cap W$ is transverse. 
Since complex points come in pairs in $X\cap W$, $\#(X\cap W)_\RR \equiv \deg X \mod 2$. So, if $(X\cap W)_\RR=\{P_1,\ldots,P_n\}$, we must have $n\equiv \deg X \mod 2$, a contradiction.

Now suppose $n+d>N$. 
The intersection $X\cap W$ contains real smooth points $P_1,\ldots,P_n$ and it has complex dimension $n+d-N>0$. 
By \cite[Proposition 7.6.2]{bochnak2013real}, the set $(X\cap W)_\RR$ has a semialgebraic neighborhood of dimension $n+d-N$ around each $P_i$ for $i=1,\ldots,n$. Hence $(X\cap W)_\RR$ contains infinitely many points.
\end{proof}

\section{Examples}
\subsection{Segre-Veronese varieties}\label{subsec: segre-veronese}

The integers $\Ncal(X)_{\min}$ and $\Ncal(X)_{\max}$ characterize $\Ncal(X)$ and the cases in the real generalized trisecant trichotomy, see Proposition \ref{thm: N(X)} and Theorem~\ref{thm: real trisecant lemma}.
In this section, we study $\Ncal(X)_{\min}$ and $\Ncal(X)_{\max}$ for Segre-Veronese varieties. We prove Theorems \ref{thm: N(X) for segre-veronese} and \ref{thm: N(X) for small segre-veronese}.

\begin{lemma}\label{lem: degree of SV}
The dimension of $\SV_{(m_1,\ldots,m_n)}(d_1,\ldots,d_n)$ is $m_1+\ldots+m_n$
and the degree is 
$$ \frac{(m_1+\ldots+m_n)!}{m_1!\cdots m_n!} \prod_{i=1}^n d_i^{m_i}.$$
\end{lemma}
\begin{proof}
This is an exercise about Hilbert polynomials; we include a proof for convenience.
The variety $\SV_{(m_1,\ldots,m_n)}(d_1,\ldots,d_n)$ has dimension $m_1+\ldots+m_n$ since it is the image of an embedding of $\PP_\CC^{m_1}\times \cdots \times \PP_\CC^{m_n}$.
Segre-Veronese varieties are toric varieties. The corresponding polytope for $\SV_{(m_1,\ldots,m_n)}(d_1,\ldots,d_n)$ is $d_1\triangle_{m_1}\times \cdots \times d_n\triangle_{m_n}$ where $d_i\triangle_{m_i}$ means the simplex of dimension $m_i$ dilated $d_i$ times.
By Kushnirenko's Theorem \cite{kushnirenko1976newton}, its degree is $\frac{(m_1+\ldots+m_n)!}{m_1!\cdots m_n!} \Vol(d_1\triangle_{m_1}\times \cdots \times d_n\triangle_{m_n})=\frac{(m_1+\ldots+m_n)!}{m_1!\cdots m_n!} \prod_{i=1}^n d_i^{m_i}$.
\end{proof}

Let $\mathbf{x}_i=[x_{i,0},\ldots,x_{i,{m_i}}]$ be the projective coordinates of $\PP_\CC^{m_i}$.
The Segre-Veronese variety $X = \SV_{(m_1,\ldots,m_n)}(d_1,\ldots,d_n)$ is the image of the monomial map that sends $(\mathbf{x}_1,\ldots,\mathbf{x}_n)$ to the vector of monomials with multidegree $(d_1,\ldots,d_n)$. 
Hence the intersection points of $X$ and a complementary dimension real linear space can be expressed as $\dim X= m_1+\ldots+m_n$ polynomials in $(\mathbf{x}_0,\ldots,\mathbf{x}_n)$ with monomials of multidegree $(d_1,\ldots,d_n)$.

We consider polynomial systems of the product form
\begin{align}
\begin{split}
    f_1^{(1)}(\fx_1) \cdot \ldots \cdot f_n^{(1)}(\fx_n)&=0\\
    \vdots&\\
    f_1^{(M)}(\fx_1)\cdot \ldots \cdot f_n^{(M)}(\fx_n)&=0,
\label{eq: poly system max}
\end{split}
\end{align}
where $M = m_1 + \cdots + m_n$.
Each polynomial consists of monomials of multidegree $(d_1,\ldots,d_n)$.
Therefore, the solutions to the polynomial system are the intersection points of $X$ with some complementary dimension real linear space.
The $\fx_i$ part of each solution is a solution to $m_i$ equations $f_i^{(j_1)}(\fx_i)=0,\ldots,f_i^{(j_{m_i})}(\fx_i)=0$ for some $1\leq j_1 <\ldots< j_{m_i}\leq M$.

 For Segre-Veronese varieties, 
the maximum number of real solutions is the degree.
\begin{lemma}\label{lem: SV max possible}
For any Segre-Veronese variety $X = \SV_{(m_1,\ldots,m_n)}(d_1,\ldots,d_n)$, we have $$\Ncal(X)_{\max}=\deg X =\frac{(m_1+\ldots+m_n)!}{m_1!\cdots m_n!} \prod_{i=1}^n d_i^{m_i}.$$
\end{lemma}

\begin{proof}
We build a system of equations of the form in~\eqref{eq: poly system max}. 
Let $M=m_1+\ldots+m_n$. Then $\dim X = M$.
We pick $d_i M$ generic vectors in $\RR^{m_i+1}$ and denote them by $\ell_1^{(1)},\ldots,\ell_1^{(d_i)},\ldots,\ell_M^{(1)},$ $\ldots,\ell_M^{(d_i)}$.
Let $f_i^{(j)}(\fx_i)=\prod_{k=1}^{d_i} (\ell_j^{(k)}\cdot \fx_i)$ for $i=1,\ldots,n, j=1,\ldots,M$. 
The polynomial $f_i^{(j)}(\fx_i)$ is homogeneous of degree $d_i$ in $\fx_i$. 

There are $\frac{M!}{m_1!\cdots m_n!}$ many ways to partition $M$ polynomials into subsets of size $(m_1,\ldots,m_n)$. 
For the subset of $m_i$ polynomials, we set their $f_i$ part equal to 0 for $i=1,\ldots,n$.
This has $d_i^{m_i}$ solutions, since the system $f_i^{(j_1)}(\fx_i)=0,\ldots,f_i^{(j_{m_i})}(\fx_i)=0$ where $1\leq j_1 <\ldots< j_{m_i}\leq M$  has $d_i^{m_i}$ real distinct solutions, 
each one corresponding to the solution of $m_i$ linear equations.
Hence, the polynomial system \eqref{eq: poly system max} has $\frac{(M)!}{m_1!\cdots m_n!} \prod_{i=1}^n d_i^{m_i}$ solutions. The solutions are distinct since the linear forms are generic.
\end{proof}

When $X$ is a toric variety, triangulations of its associated polytope give information about $\Ncal(X)_{\max}$.
The following result also covers the previous Theorem.

\begin{theorem}[{\cite[Corollary 2.4]{sturmfels1994number}}]
Suppose $X$ is a toric variety with associated polytope $P$. If $P$ admits a unimodular regular triangulation with each simplex having unit volume 1, then $\Ncal(X)_{\max}=\deg X$.
\end{theorem}

For a large class of Segre-Veronese varieties, the minimum number of real solutions is 0.

\begin{lemma}\label{lem: SV min possible all even}
Let $X = \SV_{(m_1,\ldots,m_n)}(d_1,\ldots,d_n)$. If at least one of $d_1,\ldots,d_n$ is even, then 
$\Ncal(X)_{\min}=0$.
\end{lemma}
\begin{proof}
We build a polynomial system as in~\eqref{eq: poly system max}.
Let $M=m_1+\ldots+m_n$. Then $\dim X = M$.
Without loss of generality, assume $d_1$ is even.
Pick $2M$ generic vectors in $\RR^{m_1+1}$ and denote them by $v_1^{(1)},v_1^{(2)},\ldots,v_M^{(1)},v_M^{(2)}$.
Let $f_1^{(j)}(\fx_1)=(v_j^{(1)}\cdot \fx_1)^{d_1}+(v_j^{(2)} \cdot \fx_1)^{d_1}$ for $j=1,\ldots,M$.
For a system $f_1^{(j_1)}(\fx_1)=0,\ldots,f_1^{(j_{m_1})}(\fx_1)=0$ where $1\leq j_1 <\ldots< j_{m_1}\leq M$, if it has a real solution, then the solution should satisfy 
$2m_1$ generic linear relations which is impossible. 
However, the system has $d_1^{m_1}$ distinct complex solutions since each $f_1^{(j)}(\fx_1)$ is a product of $d_1$ linear relations with complex coefficients.

Let $f_i^{(j)}(\fx_i)$ be a generic homogeneous degree $d_i$ polynomial, 
for $i=2,\ldots,n$ and $j=1,\ldots,M$.
It has $\frac{(m_1+\ldots+m_n)!}{m_1!\cdots m_n!} \prod_{i=1}^n d_i^{m_i}$ solutions. All the solutions are complex since their $\fx_1$ parts are complex and are distinct by genericity.
\end{proof}

\begin{lemma}\label{lem: reduce to multilinear}
Let $X = \SV_{(m_1,\ldots,m_n)}(d_1,\ldots,d_n)$ with all $d_i$ odd. Then 
$$\Ncal(X)\supseteq \Ncal(\SV_{(m_1,\ldots,m_n)}(1,\ldots,1)).$$
\end{lemma}

\begin{proof}
Let $M=m_1+\ldots+m_n$. Then $\dim X = M$. Suppose $k\in \Ncal(\SV_{(m_1,\ldots,m_n)}(1,\ldots,1))$ is achieved by a polynomial system $g_1=0,\ldots,g_M=0$ where each $g_i$ is a real coefficients polynomial with multidegree $(1,\ldots,1)$. 
Let $f_i^{(j)}(\fx_i)=(v_{i,j}^{(1)}\cdot \fx_i)^{d_i-1}+(v_{i,j}^{(2)} \cdot \fx_i)^{d_i-1}$ for $i=1,\ldots,n, j=1,\ldots,M$.
We consider the polynomial system 
\begin{align}
\begin{split}
    g_1 \cdot f_1^{(1)}(\fx_1) \cdot \ldots \cdot f_n^{(1)}(\fx_n)&=0\\
    \vdots&\\
    g_M \cdot f_1^{(M)}(\fx_1)\cdot \ldots \cdot f_n^{(M)}(\fx_n)&=0.
\label{eq: poly system d_i all odd reduce to 1}
\end{split}
\end{align}
Each polynomial consists of multidegree $(d_1,\ldots,d_n)$ monomials. 
Therefore, the solutions to the polynomial system are the intersection points of $X$ with some complementary dimension real linear space. 
All the solutions to \eqref{eq: poly system d_i all odd reduce to 1} are distinct by genericity of $v_{i,j}^{(1)},v_{i,j}^{(2)}$.
If there is a real solution satisfying $f_i^{(j)}(\fx_i)=0$, it must satisfy $v_{i,j}^{(1)}\cdot \fx_i=0$ and $v_{i,j}^{(2)}\cdot \fx_i=0$. So, it is a solution to a polynomial system with at least $M+1$ equations. But this contradicts $\dim X =M$ and the genericity of $v_{i,j}^{(1)},v_{i,j}^{(2)}$.
This implies all real solutions of \eqref{eq: poly system d_i all odd reduce to 1} come from $g_1=0,\ldots,g_M=0$.
Hence, $k\in \Ncal(X)$.
\end{proof}

\begin{remark}
Note that for fixed $m_1,\ldots,m_n$, the codimension of $\SV_{(m_1,\ldots,m_n)}(d_1,\ldots,d_n)$ increases as we increase $d_1,\ldots,d_n$. 
For large enough $d_1+\ldots+d_n$, we must have
$$\Ncal(\SV_{(m_1,\ldots,m_n)}(d_1,\ldots,d_n))_{\min}<\codim (\SV_{(m_1,\ldots,m_n)}(d_1,\ldots,d_n))+1.$$ Hence, if $\deg(\SV_{(m_1,\ldots,m_n)}(d_1,\ldots,d_n)) \equiv \codim (\SV_{(m_1,\ldots,m_n)}(d_1,\ldots,d_n))+1 \mod 2$, we are in the middle case of Theorem \ref{thm: real trisecant lemma}(b).
\end{remark}

\begin{lemma}\label{lem: d_i all odd}
Let $X = \SV_{(m_1,\ldots,m_n)}(d_1,\ldots,d_n)$.
If at least two of $m_1,\ldots,m_n$ are odd, then
$\Ncal(X)_{\min}=0$.
\end{lemma}
\begin{proof}
Without loss of generality, we assume that $m_1$ and $m_2$ are odd.
We first suppose that
$M=m_1+\ldots+m_n$ is even. 
Consider the polynomial system as in~\eqref{eq: poly system max}
where each $f_i^{(j)}$ is a product of $d_i$ generic linear forms with coefficients in $\CC^{m_i+1}$ and where
$f_i^{(2j-1)}$ is conjugate to $f_i^{(2j)}$ for $1\leq j \leq \frac{M}{2}$.
This system can be rewritten as a system of real polynomials by taking the real and imaginary parts of the conjugate pairs. All the solutions to this system are distinct since the linear relations are generic.
The $\fx_1$ part of each solution is the solution to $m_1$ linear relations, each a factor of $f_1^{(j_1)}(\fx_1)=0,\ldots,f_1^{(j_{m_1})}(\fx_1)=0$ for some $1\leq j_1 < \ldots<j_{m_1} \leq M$.
If it is real, then the linear relations should be a collection of conjugate pairs, but this is impossible because $m_1$ is odd.
Hence, the system \eqref{eq: poly system max} has no real solutions.

It remains to consider the case when $M$ is odd.
Consider a system of equations as in~\eqref{eq: poly system max},
where each $f_i^{(k)}$ is a product of $d_i$ generic linear forms with coefficients in $\CC^{m_i+1}$ for $k\leq M-1$, where
$f_i^{(2j-1)}$ is conjugate to $f_i^{(2j)}$ for $i=1,2$ and $1\leq j \leq \lfloor \frac{M}{2} \rfloor$,
and where each $f_i^{(M)}$ is a product of $d_i$ generic linear relations with real coefficients for $i=1,\ldots,n$. 
The $\fx_i$ part of each solution is the solution to $m_i$ linear relations that are factors of $f_i^{(j_1)}(\fx_i),\ldots,f_i^{(j_{m_1})}(\fx_i)$ for some $1\leq j_1 < \ldots<j_{m_1} \leq M$ and $i=1,\ldots,n$.
If it is real, then the linear relations should be a collection of real solutions and pairs of conjugate solutions. Hence $f_1^{(M)}(\fx_1) = 0$.
Similarly, we must have 
$f_2^{(M)}(\fx_2) = 0$.
But a solution cannot make two factors of the same polynomial in our system vanish, by genericity. 
Hence this system has no real solutions.
\end{proof}

\begin{proof}[Proof of Theorem \ref{thm: N(X) for segre-veronese}]
Part (i) is shown in Lemma \ref{lem: SV max possible}.
Part (iii) is shown in Lemma \ref{lem: SV min possible all even}.
Part (iv) is shown in Lemma \ref{lem: reduce to multilinear}.
Part (ii) follows from Lemma \ref{lem: d_i all odd}.
\end{proof}

The variety $\SV_{1,n}(1,1)$ has degree $n+1$. Hence it is a variety with minimal degree. 
The above results do not cover the case of $\Ncal(\SV_{1,n}(1,1))_{\min}$ when $n$ is even.
However, we can use induction to show the following. 

\begin{lemma}\label{thm:seg 1,n}
The Segre variety $X = \SV_{1,n}(1,1)$ has degree $n+1$ and 
$$
\Ncal(X)= \{ k: 0 \leq k\leq n+1, k\equiv n+1 \mod 2  \}.
$$
\end{lemma}

\begin{proof}
Suppose first that $n$ is odd.
By Lemma~\ref{lem: SV max possible}, $\Ncal(X)_{\max}=n+1$.
By Lemma~\ref{lem: d_i all odd},  $\Ncal(X)_{\min}=0$.
So, $\Ncal(X)= \{\, k\,: 0 \leq k\leq n+1,\,\, k\equiv n+1 \mod 2  \,\}$.

Now, suppose that $n$ is even. It suffices to show that $\Ncal(X)_{\min}=1$. 
Let the coordinates for $\PP_\CC^1$ and $\PP_\CC^n$ be $[x_0,x_1]$ and $[y_0,\ldots,y_n]$.
Consider a generic system of real polynomials:
\begin{align}
f_1(x_0,x_1,y_0,\ldots,y_{n-1})+y_n(\lambda_1x_1+\mu_1x_0)&=0\\
\vdots&\\
f_n(x_0,x_1,y_0,\ldots,y_{n-1})+y_n(\lambda_nx_1+\mu_nx_0)&=0\\
y_nx_1&=0, \label{eq:segre 1,n}
\end{align}
where 
each $f_i$ is bihomogeneous with multidegree $(1,1)$ in $[x_0,x_1]$ and $[y_0,\ldots,y_{n-1}]$,
and the system $f_1(x_0,x_1,y_0,\ldots,y_{n-1})=0,\ldots,f_n(x_0,x_1,y_1,\ldots,y_{n-1})=0$ has no real solutions (which is possible by the result when $n$ is odd).
We may further assume the system has no solution with $x_1=0$, by genericity.
Substituting $x_1=0$ in the first $n$ polynomials of \eqref{eq:segre 1,n} gives back $n$ linear equations with one real solution; we assume the solution has $y_n\neq0$, by genericity.
Thus \eqref{eq:segre 1,n} has
\[
\deg \SV_{1,n-1}(1,1)+1=\deg \SV_{1,n}(1,1)
\]
distinct solutions, of which exactly one is real.
Hence, $\Ncal(X)_{\min}=1$.
\end{proof}

We can use the result about $\Ncal(\SV_{1,n}(1,1))$ to inductively construct polynomial systems and gain information about $\Ncal(\SV_{2,n}(1,1))$.
The result of $\Ncal(\SV_{2,2})(1,1)$ is obtained using numerical algebraic geometry software \cite{breiding2018homotopycontinuation} and certified via \cite{breiding2023certifying}.
The remaining cases use a result about orbits of tensors under the action of general linear groups. 

\begin{lemma}\label{lem:seg 2,n}
For Segre varieties $X = \SV_{2,n}(1,1)$ with $n\geq 2$, we have
$$
\Ncal(X)\supseteq \{\, k:\, \lfloor \frac{n-2}{2} \rfloor \leq k\leq \deg X ,\,\, k\equiv  \deg X \mod 2  \,\}.
$$
\end{lemma}
\begin{proof}
We prove the result by induction on $n$.
When $n=2$, it can be checked using \cite{breiding2018homotopycontinuation,breiding2023certifying} that the polynomial system 
\begin{align*}
2x_0y_2 + x_1y_2 + 2x_2y_0 + x_2y_1 + x_2y_2 &=0\\
 x_0y_0 + 2x_0y_1 + 2x_0y_2 + x_1y_0 - x_1y_2 + 2x_2y_1&=0\\
 x_0y_0 + 2x_0y_1 + 2x_1y_0 - x_2y_0 - x_2y_1&=0\\
 x_0y_1 + 2x_0y_2 + 2x_1y_1 - x_1y_2 - x_2y_0 + 2x_2y_2&=0
\end{align*}
has 0 real solutions.
Therefore, $\Ncal(SV_{2,2}(1,1))_{\min}=0$. 
By Lemma \ref{lem: SV max possible} and Theorem \ref{thm: real trisecant lemma}, $\Ncal(SV_{2,2}(1,1))=\{0,2,4,6\}$.

Now, we assume the result for $n-1$. Suppose that the coordinates for $\PP_\CC^2$ and $\PP_\CC^n$ are $[x_0,x_1,x_2]$ and $[y_0,\ldots,y_n]$.
We have ${n+2\choose 1} = { n+1\choose 2}+{n+1\choose 1}$, i.e. 
$$\deg\SV_{2,n}(1,1)=\deg \SV_{2,n-1}(1,1)+\deg\SV_{1,n}(1,1).$$
We construct a system with real coefficients of the following structure
\begin{align}
\sum_{i\leq 1,j\leq n-1}\lambda_{i,j}^{(1)}x_iy_j+x_2(\sum_{j\leq n-1}\mu_{j}^{(1)} y_j)+y_n(\sum_{i\leq 1}\nu_{i}^{(1)} x_i)&=0\\
\vdots&\\
\sum_{i\leq 1,j\leq n-1}\lambda_{i,j}^{(n+1)} x_i y_j+ x_2 (\sum_{j\leq n-1}\mu_{j}^{(n+1)} y_j)+
y_n(\sum_{i\leq 1}\nu_{i}^{(n+1)} x_i)&=0\\
x_2y_n&=0
\label{eq:segre 2 times n}
\end{align}
For generic choices of coefficients, there is no solution with $x_2=y_n=0$.
So, all solutions are distinct for \eqref{eq:segre 2 times n}. And the number of real solutions is the sum of the number of real solutions for a system in $\SV_{1,n}(1,1)$ (by setting $x_2=0$) and a system in $\SV_{2,n-1}(1,1)$ (by setting $y_n=0$). 

We pick a generic system with real coefficients in $\SV_{1,n}(1,1)$ with all solutions distinct and $k$ of them real. 
We store the coefficients of $x_0y_0,\ldots,x_1y_{n-1}$ of each polynomial in $n+1$ matrices $A_1,\ldots,A_{n+1}$ of size $2\times n$.
We pick a generic system with real coefficients in $\SV_{2,n-1}(1,1)$ with all solutions distinct and $l$ of them real and we store the coefficients of $x_0y_0,\ldots,x_1y_{n-1}$ in $n+1$ matrices $B_1,
\ldots,B_{n+1}$ of size $2\times n$.
If we can find invertible matrices $P\in \GL(2),Q\in \GL(n)$ and $M\in \GL(n+1)$ such that $$P(\sum_{j=1}^{n+1} M_{i,j}A_j)Q=B_i$$ 
for $i=1,\ldots,n+1$, 
then by a change of basis on $\{x_0,x_1\}$ and $\{y_0,\ldots,y_{n-1}\}$ and some linear transformations of the polynomial equations,
we can combine the two systems in $\SV_{1,n}(1,1)$ and $\SV_{2,n-1}(1,1)$ to generate a system in $X$ of the form \eqref{eq:segre 2 times n}. 
The obtained system will have all complex solutions distinct and $k+l$ solutions real. 
By inductive hypothesis and Lemma \ref{thm:seg 1,n}, we can choose $k$ to cover all positive integers in $[0, n+1]$ with $k\equiv n-3 \mod 2$ and $\ell$ to cover all positive integers in $[\lfloor \frac{n-3}{2} \rfloor, {n+1 \choose 2}]$ with $\ell \equiv {n+1 \choose 2} \mod 2$. 
Hence $k+\ell$ covers all positive integers between $\lfloor \frac{n-3}{2} \rfloor+ \frac{1+(-1)^{n}}{2} = \lfloor \frac{n-2}{2} \rfloor$ and ${n+2 \choose 2}$ with the same 
parity as $n+2 \choose 2$.

We show the existence of $P\in \GL(2),Q\in \GL(n)$ and $M\in \GL(n+1)$ such that $P(\sum_{j=1}^{n+1} M_{i,j}A_j)Q=B_i$ for $i=1,\ldots,n+1$.
We stack the $n+1$ matrices $A_i$ to form a $2\times n\times (n+1)$ tensor $A$ and similarly we stack $B_i$ to form a tensor $B$.
The existence of $(P,Q,M)$ is equivalent to that  $A$ and $B$ are in the same orbit under the group action $G=\GL(2)\times \GL(n)\times \GL(n+1)$. 
Note that the numbers of real solutions for the selected polynomial systems in $\SV_{1,n}(1,1)$ and $\SV_{2,n-1}(1,1)$ do not change if we slightly perturb the entries of the matrices $\{A_i\}_{i=1}^{n+1}$, $\{B_i\}_{i=1}^{n+1}$.
So, the collection of matrices $\{A_i\}_{i=1}^{n+1}$ and $\{B_i\}_{i=1}^{n+1}$ can be replaced by $\{A^\prime_i\}_{i=1}^{n+1}\in \Vcal_A$, $\{B^\prime_i\}_{i=1}^{n+1}\in \Vcal_B$ in some nonempty open sets $\Vcal_A$, $\Vcal_B$ around $\{A_i\}_{i=1}^{n+1}$ and $\{B_i\}_{i=1}^{n+1}$ respectively. 
We also let $\Vcal_A,\Vcal_B$ denote the corresponding open sets around the tensors $A,B$.
So, it suffices to show that there are some tensors in $\Vcal_A$ and $\Vcal_B$ that share the same orbit under the action of $G$. 
By counting parameters $2^2+n^2+(n+1)^2-2>2n(n+1)$, so by \cite{sato1977classification} (alternatively, see \cite[Theorem 10.2.2.1]{landsberg2011tensors}), $G$ has a Zariski-dense orbit $\Vcal$ on $\CC^2\otimes \CC^{n} \otimes \CC^{n+1}$. 
So, $\Vcal\cap (\Vcal_A)_\RR$ and $\Vcal\cap (\Vcal_B)_\RR$ are nonempty. 
The result follows by picking $A \in \Vcal\cap (\Vcal_A)_\RR, B \in \Vcal\cap (\Vcal_B)_\RR$.
\end{proof}

\begin{remark}
The above induction does not apply to $\SV_{m,n}(1,1)$ with $m,n\geq 3$ since there is no longer a dense orbit of $\GL(m)\times \GL(n) \times \GL(m+n-1)$ in $\CC^{m}\otimes \CC^n \otimes \CC^{m+n-1}$.
\end{remark}

\begin{proof}[Proof of Theorem \ref{thm: N(X) for small segre-veronese}]
Part (i) is proved in Lemma \ref{thm:seg 1,n} and 
part (ii) in Lemma \ref{lem:seg 2,n}
\end{proof}

\subsection{Varieties with $\mathcal{N}(X)_{\max}$ minimal}\label{subsec: N(X) max minimal}
We say that a $d$ dimensional variety $X\subseteq \PP_\CC^{N-1}$ has $\Ncal(X)_{\max}$ minimal if $\Ncal(X)_{\max}=N-d$.
This implies that in Theorem \ref{thm: real trisecant lemma}$(b)$, the probability of $(X\cap W)_\RR=\{P_1,\ldots,P_n\}$ is 1 where $n=N-d$ and $W=\Span\{P_1,\ldots,P_n\}$ for $n$ generic points $P_1,\ldots,P_n$ on $X_\RR$.

A real variety of minimal degree with a smooth point has $\Ncal(X)_{\max}$ minimal, e.g. the Segre Veronese varieties $\SV_{1,n}(1,1)$, see Lemma \ref{thm:seg 1,n}.
It is a natural question whether there are real varieties of non-minimal degrees and $\Ncal(X)_{\max}$ minimal. 
We give examples of plane curves and hypersurfaces.
\begin{lemma}
Suppose $X$ is a smooth real plane curve with a smooth real point.
Then $\Ncal(X)_{\max}$ is minimal if and only if $X_\RR$ is a convex oval. 
\end{lemma}
\begin{proof}
When $X$ is a real plane curve with a smooth real point, $\Ncal(X)_{\max}$ is minimal when a line in $\PP_\RR^2$ intersects $X$ in at most two real points.

Assuming that $X_\RR$ is a convex oval, it immediately follows that $\Ncal(X)_{\max}$ is minimal.
Now, we suppose $\Ncal(X)_{\max}$ is minimal.
If $\deg X$ is even, $X_\RR$ consists of ovals and if $\deg X$ is odd, $X_\RR$ consists of ovals and one pseudoline.
In the first case, if we have at least two ovals, we can choose a line passing through two interior points of the two ovals, resulting in at least four intersection points with $X_\RR$. 
In the second case, we can choose a line passing through an interior point of one oval. The line will also intersect the pseudoline, resulting in at least three intersection points with $X_\RR$.
So, $X_\RR$ consists of one oval. 
The oval must be convex otherwise there would be a line intersecting $X$ in more than two real points.
\end{proof}

There exists a plane curve $X$ with $\Ncal(X)_{\max}$ minimal for every even degree.
\begin{example}
For $d\in \NN_+$, the real part of the curve $X_d: x^{2d}+y^{2d}=z^{2d}$ has $\Ncal(X_d)_{\max}$ minimal. 
Since the curve $X_d$ does not intersect $z=0$, we will also denote the plane curve $x^{2d}+y^{2d}=1$ by $X_d$.
The curve $X_d$ has one oval contained in the unit square $[0,1]^2$.
Define $F(x,y) = x^{2d}+y^{2d}-1$.
The curvature of $X_d$ at a point $(x,y)$ is 
$$
\kappa= \frac{F_y^2F_{xx}-2F_xF_yF_{xy}+F_x^2F_{yy}}{(F_x^2+F_y^2)^{\frac{3}{2}}}=(2d-1)\frac{y^{4d-2}x^{2d-2}+x^{4d-2}y^{2d-2}}{(x^{4d-2}+y^{4d-2})^{\frac{3}{2}}}>0.
$$
The curvature is always positive, so $X_d$ has a convex oval, and $\Ncal(X_d)_{\max}$ is minimal.
\end{example}

With similar ideas, we can construct hypersurfaces $X$ with $\Ncal(X)_{\max}$ minimal for every even degree.

\begin{example}\label{ex: 2ne open question}
For $d\in \NN_+$, the hypersurface 
$X_d: x_1^{2d}+\ldots+x_n^{2d}=x_0^{2d}$ has $\Ncal(X_d)_{\max}$ minimal. 
Since $X_d$ does not intersect $x_0=0$, by an abuse of notation, we will also denote the affine hypersurface $x_1^{2d}+\ldots+x_n^{2d}=1$ by $X_d$.
The real part of $X_d$ has one connected component, so it is enough to show that it is convex.
Suppose that $\mathbf{a}=(a_1,\ldots,a_n)$ and $\mathbf{b}=(b_1,\ldots,b_n)$ are two distinct points on $X_d$, then we need to show that 
$$(\lambda a_1+(1-\lambda)b_1)^{2d}+\ldots+(\lambda a_n+(1-\lambda)b_n)^{2d}< 1,
$$
for $\lambda\in (0,1)$.
By Hölder's inequality (see \cite[Equation 2.7.2]{hardy1952inequalities}), 
$$(\lambda a_i^{2d}+(1-\lambda){b_i}^{2d})^{\frac{1}{2d}}(\lambda+1-\lambda)^{\frac{2d-1}{2d}}
\geq 
\lambda^{\frac{1}{2d}}|a_i|\lambda^{\frac{2d-1}{2d}}+
(1-\lambda)^{\frac{1}{2d}}|b_i| (1-\lambda)^{\frac{2d-1}{2d}} 
\geq |\lambda a_i+(1-\lambda)b_i| $$ and the equality holds if and only if $a_i=b_i$.
Hence, 
$$
(\lambda a_1+(1-\lambda)b_1)^{2d}+\ldots+(\lambda a_n+(1-\lambda)b_n)^{2d}
< \sum_{i=1}^n \bigl(\lambda a_i^{2d}+(1-\lambda)b_i^{2d} \bigl)\ = 1.$$
\end{example}

There are curves $X\subseteq \PP_{\CC}^3$ of infinitely many possible degrees with $\Ncal(X)_{\max}$ minimal.

\begin{example}[Kummer, Manevich \cite{kummer2021huisman}
]\label{ex: 2 open question} Let $X$ be the curve in $\PP_\CC^3$ of degree $k+2e+1$ defined in \cite[Construction 1]{kummer2021huisman}.
It has $\Ncal(X) = \{k-1,k+1\}$. 
Take $k=2$. 
Then $X$ has $\Ncal(X)_{\max} = 3 = \codim X +1$  minimal.
Since $e$ can take any positive integer value, there is a curve $X$ in $\PP_\CC^3$ with $\deg X\geq 3$ odd and $\Ncal(X)_{\max}$ minimal.
\end{example}

 We pose the following open problem.
 \begin{question} 
 Does there exist, for every triple of positive integers $d,n,m$, with $d>n+1>2$, $d\equiv n+1 \mod 2$, and $m>1$, an irreducible, smooth, non-degenerate $m$-dimensional real projective variety $X$ of degree $d$ with $\codim X=n>1$ and $\Ncal(X)_{\max}=\codim X+1=n+1$ minimal?
 \end{question}

\section{Applications}\label{sec: applications}
In this section, we discuss applications of the real generalized trisecant trichotomy Theorem~\ref{thm: real trisecant lemma} and our characterization of the set of possible real intersection points Proposition~\ref{thm: N(X)} to independent component analysis, tensor decompositions and the study of typical tensor ranks.

\subsection{Independent component analysis (ICA)}

ICA writes observed variables as linear mixtures of independent sources. That is,
\begin{equation}
\label{eqn:defica}
\mathbf{x}=A\mathbf{s},
\end{equation}
where $\mathbf{s} = (s_1,\ldots,s_J)\T$ is a vector of independent sources, $\mathbf{x}=(x_1,\ldots,x_I)\T$ collects the observed variables, 
and $A \in \RR^{I \times J}$ is an unknown mixing matrix. The ICA model is said to be identifiable if the mixing matrix $A$ can be uniquely recovered, up to scaling and permutation of its columns.

In \cite{wang2024identifiability}, 
a matrix $A\in \RR^{I\times J}$ is called identifiable if for any vector of source variables $\mathbf{s}=(s_1,\ldots,s_J)$ with at most one Gaussian source, one can recover $A$ uniquely up to column scaling and permutation from observation $A\mathbf{s}$.
This translates to the following algebraic geometric criterion.

\begin{theorem}[{\cite[Theorem 1.5]{wang2024identifiability}}]
Fix $A\in \RR^{I\times J}$ with columns $\fa_{1},\ldots, \fa_{J}$ and no pair of columns collinear. Then $A$ is identifiable if and only if the linear span of $\fa_1^{\otimes 2},\ldots, \fa_J^{ \otimes 2}$ does not contain any real matrix $\mathbf{b}^{\otimes 2}$ unless $\mathbf{b}$ is collinear to $\fa_j$ for some $j \in \{ 1, \ldots, J\}$.
\end{theorem}

This result poses the question: when does a linear space spanned by $J$ points in $\SV_{I-1}(2)$, the second Veronese embedding of $\PP_\CC^{I-1}$, 
intersect the linear space in exactly these $J$ points.
Hence, one obtains a complete classification of generic identifiability via the real generalized trisecant trichotomy on second Veronese varieties.

\begin{theorem}[{\cite[Theorem 1.9]{wang2024identifiability}}]
Let $A \in \RR^{I \times J}$ be generic. Then
\begin{enumerate}[(i)]
    \item If $J \leq {I \choose 2}$ 
    or if $(I, J) = (2, 2)$ or $(3, 4)$, then $A$ is identifiable;
    \item If $J= {I\choose 2}+1$, where $I\geq 4$ and $I \equiv 2,3 \mod 4$, then there is a positive probability that $A$ is identifiable and a positive probability that $A$ is non-identifiable;
    \item If $J> {I\choose 2}+1$ or if $J= {I\choose 2}+1$ and $I\equiv 0,1 \mod 4$, then $A$ is non-identifiable. 
\end{enumerate}
\end{theorem}

\subsection{Uniqueness of tensor decompositions}
Consider a generic partially symmetric tensor 
$$T=\sum_{j=1}^J \lambda_j (\mathbf{v}_1^{(j)})^{\otimes 2d_1}\otimes \cdots \otimes (\mathbf{v}_n^{(j)})^{\otimes 2d_n}$$ 
of rank $J$ in $\Sym_{2d_1}\RR^{m_1+1}\otimes \cdots \otimes \Sym_{2d_n}\RR^{m_n+1}$.
We flatten \(T\) into a square matrix
$M$ by grouping its indices into two blocks of equal size.  
Each rank one term $(\mathbf{v}_1^{(j)})^{\otimes 2d_1}\otimes \cdots \otimes (\mathbf{v}_n^{(j)})^{\otimes 2d_n}$ is flattened into the rank one matrix
$\Vect(T_j)\Vect(T_j)\T,$
where 
$T_j=(\mathbf{v}_1^{(j)})^{\otimes d_1}\otimes \cdots \otimes (\mathbf{v}_n^{(j)})^{\otimes d_n}$ for $j=1,\ldots,J$, $(\cdot )\T$ denotes the transpose and \(\Vect(\cdot)\) is the \emph{vectorization operator} that maps a tensor to a column vector by stacking its entries in lexicographic order. 
Thus
$$
M = \sum_{j=1}^J \lambda_j \Vect(T_j) \Vect(T_j)\T.
$$
If $J\leq \dim \Sym_{d_1}\RR^{m_1+1}\otimes \cdots \otimes \Sym_{d_n}\RR^{m_n+1}$, the linear space 
$$W:=\Span\{ \Vect(T_1),\ldots,\Vect(T_J)\}$$ is the column span of $M$.
We also use $W$ to denote $\Span\{ T_1,\ldots,T_J\}$. 

The real part of the Segre-Veronese embedding $\SV_{m_1,\ldots,m_n}(d_1,\ldots,d_n)$  is the projectivization of the space of partially symmetric tensors $\Sym_{d_1}\RR^{m_1+1}\otimes \cdots \otimes \Sym_{d_n}\RR^{m_n+1}$.
The tensor $T$ has a unique real tensor decomposition if 
\begin{align}
(W\cap \SV_{m_1,\ldots,m_n}(d_1,\ldots,d_n))_\RR=\{T_1,\ldots,T_J\}.
\label{eq:unique decomp condition}
\end{align}
The tensor $T$ has a unique tensor decomposition when $J<{m_1+d_1 \choose d_1}\cdots {m_n + d_n \choose d_n}-(m_1+\ldots+m_n)$, by Theorem \ref{thm: real trisecant lemma}.
When $J={m_1+d_1 \choose d_1}\cdots {m_n + d_n \choose d_n}-(m_1+\ldots+m_n)$, the equality
\eqref{eq:unique decomp condition} occurs with positive probability 
if $J\equiv \frac{(m_1+\ldots+m_n)!}{m_1!\cdots m_n!} \prod_{i=1}^n d_i^{m_i} \mod 2$ and 
$\SV_{m_1,\ldots,m_n}(d_1,\ldots,d_n)$ is one of the cases in Theorem \ref{thm: N(X) for segre-veronese}.
This generalizes \cite[Proposition 3.2]{kileel2025subspace} from symmetric tensors to partially symmetric tensors.

\subsection{Typical Tensor Ranks}

In a space of tensors with certain size and order, 
an integer $r$ is called a typical rank if for a tensor $T$ with Gaussian entries, we have $\mathbb{P}\{\rank(T) = r\} > 0$.
In \cite{breiding2025typical}, the authors relate the typical ranks of a tensor in $\RR^m\otimes \RR^n \otimes \RR^
\ell$ to the intersection of a dimension $\ell-1$ linear space with the Segre variety $\SV_{(m-1,n-1)}(1,1)$.

\begin{theorem}[{\cite[Theorems 1.2 and 1.3]{breiding2025typical}}]
\begin{enumerate}[(i)]
\item 
Suppose that $(m-1)(n-1)+1\leq \ell \leq mn$. The typical ranks of a tensor in $\RR^m\otimes \RR^n \otimes \RR^\ell$ are contained in $\{\ell,\ell+1\}$ and $\ell$ is always a typical rank.
\item
For $\ell=(m-1)(n-1)+1$, $\ell+1$ is a typical rank if there is a $\ell-1$ dimensional real projective linear space that intersects the Segre variety $\SV_{m-1,n-1}(1,1)$ transversely in less than $\ell$ real points.
\item
For $\ell>(m-1)(n-1)+1$, $\ell+1$ is a typical rank if the intersection of some $\ell-1$ dimensional real projective linear space and the Segre varierty $\SV_{m-1,n-1}(1,1)$contains only complex points.
\end{enumerate}
\end{theorem}

We can generalize the result to partially symmetric tensors with multidegree $(d_1,\ldots,d_n,1)$, by Theorem \ref{thm: N(X) for segre-veronese}.

\begin{theorem}
Suppose ${m_1+d_1 \choose d_1}\cdots {m_n + d_n \choose d_n}-(m_1+\ldots+m_n)+1\leq \ell \leq {m_1+d_1 \choose d_1}\cdots {m_n + d_n \choose d_n}$. 
\begin{enumerate}[(i)]
\item 
The typical ranks of the tensors $\Sym_{d_1} \RR^{m_1+1}\times \cdots \times \Sym_{d_n} \RR^{m_n+1} \times \RR^\ell$  are contained in $\{\ell,\ell+1\}$ and $\ell$ is always a typical rank. 
\item
When $\ell={m_1+d_1 \choose d_1}\cdots {m_n + d_n \choose d_n}-(m_1+\ldots+m_n)+1$ and $\Ncal(\SV_{(m_1,\ldots,m_n)}(1,\ldots,1))_{\min}<\ell$, then $\ell+1$ is a typical rank. 
In particular, this happens in the following scenarios:
\begin{enumerate}[(a)]
    \item one of $d_i$ is even;
    \item all $d_i$ are odd, at least one of $m_1,\ldots,m_n$ is odd when $m_1+\ldots+m_n$ is even or at least two of $m_1\ldots,m_n$ are odd when $m_1+\ldots+m_n$ is odd;
    \item $n=2,d_1=d_2=1$ and $\min\{m_1,m_2\}\in \{1,2\}$. 
\end{enumerate}
\end{enumerate}
\end{theorem}
\begin{proof}
Let $X=\SV_{(m_1,\ldots,m_n)}(d_1,\ldots,d_n)$.
The number ${m_1+d_1 \choose d_1}\cdots {m_n + d_n \choose d_n}-(m_1+\ldots+m_n)$ is the codimension of $X$ in its ambient space.
Let $T$ be a generic tensor in $\Sym_{d_1} \RR^{m_1+1}\times \cdots \times \Sym_{d_n} \RR^{m_n+1} \times \RR^\ell$ with slices $T_1,\ldots,T_\ell$ in $\Sym_{d_1} \RR^{m_1+1}\times \cdots \times \Sym_{d_n} \RR^{m_n+1}$.
Let $W = \langle T_1,\ldots, T_\ell \rangle$ be the linear span of the slices. 
With probability one, $\dim W =\ell-1$ and so the rank of T is at least $\ell$ by \cite[Theorem 2.4]{friedland2012generic} (see also \cite[Theorem 2.1]{breiding2025typical}). 
So, typical ranks are bounded below by $\ell$.
If $W \cap X$ contains at least $\ell$ real points, the rank of T is $\ell$. 
Otherwise, $\ell$ is not the rank of $T$ but $W'=W\oplus \langle p \rangle$ intersects $X$ in infinitely many real points by Theorem \ref{thm: real trisecant lemma}$(c)$
for some generic $p\in X$ and $W'$ is spanned by real points in $X$.
Hence, $\rank T\leq \ell+1$

Now suppose $\ell={m_1+d_1 \choose d_1}\cdots {m_n + d_n \choose d_n}-(m_1+\ldots+m_n)+1$.
The number $\ell+1$ is a typical rank when there is a positive probability for a dimension $\ell-1$ real linear space to intersect $\SV_{(m_1,\ldots,m_n)}(d_1,\ldots,d_n)$ in less than $\ell$ real points.
This happens precisely when $\Ncal(X)_{\min}<\ell$.
By Theorem \ref{thm: N(X) for segre-veronese},  we have $\Ncal(X)_{\min}=0$ for the first two cases listed above. 
For the third case, if $\min\{m_1,m_2\}=1$ then $\Ncal(X)_{\min}=0$ or 1. 
If $\min\{m_1,m_2\}=2$ and without loss of generality $m_1=2$, then $\Ncal(X)_{\min} \leq  \lfloor \frac{m_2+2}{2} \rfloor <\ell=2m_2+1$.
\end{proof}

{\bf Acknowledgements. } 
We thank Greg Blekherman, 
Alexander Esterov, Ilia Itenberg, Claus Scheiderer, Boris Shapiro, Frank Sottile, and  Bernd Sturmfels for helpful comments and discussions. 
We thank Mario Kummer for pointing us to Examples \ref{ex: open question 1} and \ref{ex: 2 open question} from \cite{kummer2021huisman,kummer2022hyperbolic}.
We thank the organizers of the ICERM workshop `Connecting Higher-Order Statistics and Symmetric Tensors', where the initial discussions on the project took place.

\bibliographystyle{alpha}
\bibliography{reference}
\end{document}